\documentclass{amsart}
\usepackage[utf8]{inputenc}
\usepackage{amsmath}
\usepackage{amsfonts} 
\usepackage{amssymb}
\newtheorem{theorem}{Theorem}[section]
\newtheorem{lemma}[theorem]{Lemma}
\theoremstyle{definition}
\newtheorem{definition}[theorem]{Definition}
\newtheorem{corollary}[theorem]{Corollary}
\numberwithin{equation}{section}
\begin{document}
\title[P-adic Line Integrals]{P-adic Line Integrals and Cauchy's Theorems}
\author{Jack Diamond}
\address{Queens College, CUNY\\Flushing\\ NY\\ USA (ret.)}
\email{jdiamond@qc.cuny.edu}
\subjclass[2010]{Primary: 11S80}
\keywords{p-adic line integral, p-adic Cauchy's Theorem, p-adic analysis}
\begin{abstract}
Working in the p-adic analog of the complex numbers, we'll define a line integral on a small arc of a circle. This allows new versions of the Residue Theorem, the Cauchy-Goursat Theorem on discs with and without holes, Cauchy's Integral Formula and the Z-P Theorem. In contrast to results in complex analysis, these integrals allow the points on a boundary circle, the bulk of a p-adic disc, to be treated the same as points interior to the boundary circle. The theory of the integral is developed, especially for functions holomorphic on an open disc, and integrals will be calculated for rational functions, Krasner analytic functions and some well-known functions that are not Krasner analytic. Some computations will produce values of Kubota-Leopoldt L-functions at ordinary integers.
\end{abstract}
\maketitle
 
\section{Introduction}\label{S:intro}
In the complex plane a closed disc of radius $R$ is thought of as being made up mostly of it's interior, an open disc of radius $R$. In $\mathbb{C}_{p}$ we have the reverse situation.  The interior of a boundary circle of radius $R$ is again an open disc of radius $R$, but the boundary circle itself consists of countably many disjoint open discs of radius $R$. The closed disc in $\mathbb{C}_{p}$ is mostly on a boundary circle.\\

If we think of an arc of a circle as the set of points on the circle that are within a given distance of some point on the circle, then an open disc of radius $R$ on a p-adic circle of radius $R$ can be called an arc.\\

These remarks suggest that a natural p-adic analog of the Cauchy Integral Theorem uses the values of a function only on one arc to obtain the values of the function on the remainder of a closed disc. Such a theorem is in section 6.\\

Section 2 defines an integral on an arc of a circle and gives some examples and basic properties.\\

Section 3 develops the theory of the integral for functions holomorphic on an arc. The concept of \emph{controlled coefficients} is introduced. The theorems show that the integral is an analog of a complex line integral.\\

Section 4 provides calculations of integrals of rational and Krasner analytic function and some familiar non-Krasner analytic functions. More generally, there is a sufficiency theorem for the existence of an integral and p-adic versions of integration by parts, a substitution theorem and a Z - P Theorem that allows zeros and poles on a boundary circle. The logarithmic derivative of the Artin-Hasse exponential function is shown to be integrable with a non-zero result.\\

Section 5 introduces and develops the concept of \emph{ray convergence}. It concludes with the calculation of an integral that gives values of Kubota-Leopoldt L-functions at positive integers\\

Sections 6 contains p-adic versions of the Residue Theorem, the Cauchy-Goursat Theorem on discs with and without holes and Cauchy's Integral Formula on discs with and without holes. These theorems are all broader than the complex versions because a boundary circle is treated much like the interior.

The idea of a p-adic line integral is not new. In 1938, L.~G. Shnirel'man published, \cite{lS1938}, a p-adic line integral that "went around" a boundary circle. The line integral presented here arises from changing a basic premise in Shnirel'man's approach .\\

\noindent\textbf{Notation}\\
$\mathbb{C}_{p}$ is the completion of an algebraic closure of $\mathbb{Q}_{p}$ with the norm  $|p|_{p} = 1/p$.\\
$p$ will be assumed to be an odd prime unless otherwise explicitly stated.\\
$D^{+}(a,R)$ is the closed disc with center $a$ and radius $R$.\\$D^{-}(a,R)$ is the open disc with center $a$ and radius $R$.\\
A function $f(x)$ on a disc $D$ will be called \textit{holomorphic} if the values of $f(x)$ can be given by a single convergent power series on $D$. \textit{H(D)} is the space of holomorphic functions on $D$.\\
$[x]$ is the greatest integer $\leq x$.\\
$\omega(x)$, defined for $x \in \mathbb{C}_{p}$ with $|x|_{p} \leq 1$, is the Teichm\"{u}ller charactor extended by $\omega(x) = 0$ if $|x|_{p} < 1$. When $|x |_{p} = 1$, $\omega(x)$ is the unique $n$-th root of unity, $(n,p) = 1$, satisfying $|x - \omega(x)|_{p} < 1$.\\
When $x$ is real and positive, $\log_{p}x$ is the real logarithm of $x$ base $p$ and $\ln x$ is the natural log of $x$. When $x \in \mathbb{C}_{p}$ and $x \neq 0$, $\log x$ is the standard, Iwasawa p-adic logarithm.

 \section{Definition of a Line Integral on an Arc}\label{S:int def} 
We'll start with a preliminary definition and some examples. Then a general definition will be given.
\begin{definition}\label{D:arc}
Given a circle in $\mathbb{C}_{p}$ with center $a$ and radius $R$ and a point $b$ on the circle, the $\emph{arc}\ \mathcal{A}=\mathcal{A}(a,b)$ at $b$ is the open disc $\{\,x \mid |x-b|_p < R\}$.
\end{definition}

\begin{definition}\label{D:path seq}
A function $\varphi(k)$ will be called a \emph{path sequence} if there is a $k_{0} \in \mathbb{Z}^{+}$ such that $\varphi(k)$ is defined for $k \in \mathbb{Z}^{+}$ with $k \geq k_{0}$, $\varphi(k) \in \mathbb{Z}^{+}$ and $\varphi(k)$  is a strictly increasing function for $k \geq k_{0}$.
\end{definition}

There's seldom a need to reference the $k_{0}$ in the definiton of \emph{path sequence} explicitly; the key fact being that all sufficiently large $k$ are in the domain of $\varphi$.
$\varphi$ will be considered  as either a function or a sequence, depending on context.
\begin{definition}\label{D:A_k} Let $f \colon \mathcal{A}(a,b) \to \mathbb{C}_{p} $. When $\varphi(k)$ is defined, $A_{f,\varphi}(k)$ is defined by
\begin{equation}\label{E:A_k}
	 A_{f,\varphi}(k)= p^{-\varphi(k)}\sum_{x}(x-a)f(x),
\end{equation}
where $x=a+(b-a)\zeta$ and $\zeta$ runs through the $p^{\varphi(k)}$-th roots of unity.
\end{definition}
\begin{definition}\label{D:1st int} Given $\mathcal{A}$, $f$ and $\varphi$, the line integral of $f$ on $\mathcal{A}$ with respect to $\varphi$ is defined by
\begin{equation}\label{E:1st int}
	\int_{\mathcal{A},\varphi}f(x)\,dx = lim_{k\to \infty} A_{f,\varphi}(k),
\end{equation}
if the limit exists.
\end{definition}
The next results follow immediately from the definition.
\begin{theorem}\label{T:subseq1}
	Suppose that for the path sequences $\varphi_{1}(k)$ and $\varphi_{2}(k)$ there is a positive integer $k_{1}$ such that $\varphi_{1}(k)$, $k \geq k_{1}$, is a subsequence of $\varphi_{2}(k)$ and 
\[
	\int_{\mathcal{A},\varphi_{2}}f(x)\,dx = L,
\]
then
\[
	\int_{\mathcal{A},\varphi_{1}}f(x)\,dx = L.
\]
\end{theorem}
\begin{theorem}\label{T:alpha1=alpha2}
	Let $\varphi_{i}(k) = \alpha_{i} + \lambda k$, $\alpha_{i} \in \mathbb{Z}$, $\lambda  \in \mathbb{Z}^{+}$, with $i = 1$,$2$. Suppose that $\alpha_{1} \equiv \alpha_{2} \pmod{\lambda}$. Then if
\[
	\int_{\mathcal{A},\varphi_{2}}f(x)\,dx \qquad \text{exists,}
\]
	the integral with $\varphi_{1}$ will exist and 
\[
	\int_{\mathcal{A},\varphi_{1}}f(x)\,dx = \int_{\mathcal{A},\varphi_{2}}f(x)\,dx.
\]
\end{theorem}
\begin{theorem}\label{T:varphi=k}
	Let $\varphi_{0}(k) = k$, $k \in \mathbb{Z}^{+}$. If
\[
	\int_{\mathcal{A},\varphi_{0}}f(x)\,dx 
\]
exists, then
\[
	\int_{\mathcal{A},\varphi}f(x)\,dx  = \int_{\mathcal{A},\varphi_{0}}f(x)\,dx.
\]
for any path sequence $\varphi$.
\end{theorem}
 \noindent Here are some examples of integrals. Details and proofs will come later.
\begin{enumerate}
\item If $f(x)$ is holomorphic on the closed disc $D^{+}(a,|a-b|_{p})$, then
\[
	\int_{\mathcal{A}(a,b),\varphi}f(x)\,dx = 0 
\]
	for any $\mathcal{A}$ and any $\varphi$;
\item If $\log x$ is the usual (Iwasawa) p-adic log function,
\[
	\int_{\mathcal{A}(1,0),k}log(1-x)\,dx = 0;
\]
\item If $x_0$ is interior to a circle, then for any $\mathcal{A}$ on the circle and any $\varphi$.
\[
	\int_{\mathcal{A},\varphi}\frac{1}{x-x_{0}}\,dx = 1;
\]
\item If $x_0$ is on the circle, but not in $\mathcal{A}$ and $\varphi(k)=\lambda k$ for suitable $\lambda$, there will be a corresponding root of unity such that
\[
	\int_{\mathcal{A},\varphi}\frac{1}{x-x_{0}}\,dx = \frac{1}{1-\zeta};
\]
\item 
\[
	\int_{\mathcal{A}(1,0),k}\left(\sum_{n\geq 0}x^{p^{n}-1}\right)\,dx \qquad \text{exists and is near \ -1.}
\]	
\end{enumerate}

With the definition above, the set of integrable, holomorphic functions on $\mathcal{A}$, for a given $\varphi$, is a vector space closed with respect to pointwise uniform convergence. However, the union of two such spaces with differing $\varphi$ may not be a vector space and may have different values for the integral of a given function.
In order to better deal with this issue and have a larger set of integrable functions, we'll define an \emph{interlocked} family of path sequences and give a broader definition of the integral on an arc.

\begin{definition}\label{D:interlocked} A non-empty family of path sequences, $\Phi $, will be called \emph{interlocked} if, given $\varphi_{1}$, $\varphi_{2} \in \Phi$, there is a $\varphi_{3} \in \Phi$ and a $k_{0} \in \mathbb{Z}^{+}$ such that $\varphi_{3}(k), k \geq k_{0}$, is a subsequence of both $\varphi_{1}(k)$ and $\varphi_{2}(k)$.
\end{definition}

Note that if $\Phi$ consists of a single path sequence, it is \emph{interlocked}.

Now we can give the general definition of the integral on an arc $\mathcal{A}$ with respect to an interlocked family of path sequences $\Phi$.
\begin{definition}[Definition of a line integral]\label{D:integral} Suppose $f \colon \mathcal{A} \to \mathbb{C}_{p}$ and $\Phi$ is an interlocked family of path sequences. Then
\[
	\int_{\mathcal{A},\Phi}f(x)\,dx = L
\]
if there is an $L \in \mathbb{C}_{p}$ such that given any $\varepsilon > 0$, there is a $K_{\varepsilon} \in \mathbb{Z}^{+}$ and a $\varphi_{\varepsilon} \in \Phi$ such that if $k > K_{\varepsilon}$ then
\begin{equation}\label{E:intdef}
	 \left| A_{f,\varphi_{\varepsilon}}(k) - L \right|_{p} < \varepsilon.
\end{equation}
\end{definition}

It's necessary to prove that given a function $f(x)$, the value of its integral is uniquely defined.
\begin{proof}
	Suppose that $L_{1}$ and $L_{2}$ satisfy the conditions of Definition \ref{D:integral}. 

Let $\varepsilon > 0$. For $i=1,2$ choose $\varphi_{i} \in \Phi$ and $K_{i}$ so that 
\[
	\left| A_{f,\varphi_{i}}(k) - L_{i} \right|_{p} < \varepsilon \ \text{for} \ k >K_{i}.
\]	

Now let $\varphi(k) \in \Phi$ be a common subsequence of the $\varphi_{i}$ for large $k$ and choose $K_{0}$, $K_{1}'$ and $K_{2}'$  with $K_{1}' > K_{1}$ and $K_{2}' > K_{2}$ such that
\[
	\varphi(K_{0}) = \varphi_{i}(K_{i}') \qquad  i = 1,2.
\]
Noting that $ A_{f,\varphi}(K_{0}) =  A_{f,\varphi_{i}}(K_{i}')$, $i = 1$, $2$, it follows that
\[
	\left| L_{1} - L_{2} \right|_{p} = \left| L_{1} - A_{f,\varphi}(K_{0}) + A_{f,\varphi}(K_{0}) - L_{2} \right|_{p} < \varepsilon.
\]
Since $\varepsilon$ could be any positive number, we must have $L_{1} = L_{2}$.
\end{proof}
If $\Phi$ consists of a single path sequence, this definition of the integral of $f$ coincides with the simpler Definition \ref{D:1st int} given earlier.\\
Theorem \ref{T:varphi=k} can be easily generalized to interlocked families.
\begin{theorem}\label{T:varphi=kPhi}
	Let $\varphi_{0}(k) = k$. If
\[
	\int_{\mathcal{A},\varphi_{0}}f(x)\,dx 
\]
exists, then
\[
	\int_{\mathcal{A},\Phi}f(x)\,dx  = \int_{\mathcal{A},\varphi_{0}}f(x)\,dx
\]
for any interlocked family $\Phi$.
\end{theorem}
\begin{proof}
Let $\varphi \in \Phi$. Apply Theorem \ref{T:varphi=k} to $\varphi$.
\end{proof}

The next two results will help find some interlocked families of path sequences.
\begin{theorem}\label{T:add alphaPhi}
	If $\Phi$ is an interlocked family of path sequences, and $\alpha \in \mathbb{Z}$, then
\[
	\Psi =  \{\,\psi \mid \psi(k) = \alpha + \varphi(k), \varphi \in \Phi \}
\]
is interlocked.
\end{theorem}
The proof is immediate from the definition.
\begin{theorem}\label{T:Phi compos}
	Suppose a family $\Phi$ of path sequences satisfies the following condition:
Given $\varphi_{1}$, $\varphi_{2} \in \Phi$, there is a $\varphi_{3}(k) \in \Phi$ and a $k_{0}$, depending on $\varphi_{1}$, $\varphi_{2}$, such that for $k \geq k_{0}$,
\[
\varphi_{3}(k) = \varphi_{1} \circ \varphi_{2}(k) = \varphi_{2} \circ \varphi_{1}(k).
\]
Then $\Phi$ is interlocked.
\end{theorem}
\begin{proof}
	$\varphi_{3}(k)$, $k \geq k_{0}$, is the desired common subsequence.
\end{proof}

An important example of an interlocked family of path sequences is
\[
	\Phi_{0} = \{\,\varphi \mid \varphi(k) = \lambda k, \ \ \lambda =1,2,\dots \}.
\]

The importance of this family is that it allows all rational functions having no poles in $\mathcal{A}$ to be integrated.

If $\alpha \in \mathbb{Z}$, then
\[
	\Phi_{\alpha} = \{\,\varphi \mid \varphi(k) = \alpha + \lambda k, \ \ \lambda = 1,2,\dots \}
\]
is also an interlocked family. The $\Phi_{\alpha}$ also allow all rational functions with no poles in $\mathcal{A}$ to be integrated. The calculations for the exact values of these integrals is in section 4.

Another interlocked family is
\[
	\Psi_{\alpha} = \{\,\varphi \mid \varphi(k) = \alpha + k^{\mu},\ \mu = 1,2,\dots \}.
\]
That these families are in fact interlocked follows easily from Theorems \ref{T:add alphaPhi} and \ref{T:Phi compos}.
The following two theorems show that commutativity of composition is not easily attained by sets of (polynomial) path functions.
\begin{theorem}
	Let $\varphi_{0}(k) = \lambda_{0} k$, $\lambda_{0} \in \mathbb{Z}^{+}$, $\lambda_{0} \geq 2$. Let $\varphi(k)$ be a polynomial such that $\varphi_{0} \circ \varphi(k) = \varphi \circ \varphi_{0} (k)$. Then, $\varphi(k) = \lambda k$ for some $\lambda \in \mathbb{Z}$. If $\varphi(k)$ is a path sequence, then $\lambda \in \mathbb{Z}^{+}$.
\end{theorem}
\begin{proof}
A comparison of the terms of the composed functions show $\varphi(k)$ has the form $\lambda k$.
\end{proof}
\begin{theorem}\label{T:mu0}
	Let $\varphi_{0}(k) = k^{\mu_{0}}$, $\mu_{0} \in \mathbb{Z}^{+}$, $\mu_{0} \geq 2$. Let $\varphi(k)$ be a polynomial and a path sequence. Suppose $\varphi_{0} \circ \varphi(k) = \varphi \circ \varphi_{0} (k)$. Then $\varphi(k) = k^{\mu}$ for some $\mu \in \mathbb{Z}^{+}$.
\end{theorem}
\begin{proof}
The following lemma will be helpful:
\begin{lemma}\label{L:comp-commmu}
	If $\varphi(k) = k^{\mu}$, $\mu \in \mathbb{Z}^{+}$, $\mu \geq 2$, h(k) is a non-constant polynomial and $\varphi \circ h(k) = h \circ \varphi (k)$, then $h(0) = 0$.
\end{lemma}
\begin{proof}
A comparison of the constant terms and the terms of lowest exponent in the two compositions will yield the result.
\end{proof}
Now suppose the conditions of Theorem \ref{T:mu0} are met. Then Lemma \ref{L:comp-commmu} applies and shows that $\varphi(0) = 0$. $\varphi(k)$ can be written
\[
	\varphi(k) = k^{\mu}h(k),
\]
where $\mu \in \mathbb{Z}^{+}$, $h(k)$ is a polynomial and $h(0) \neq 0$.
Equating the compositions of $\varphi_{0}$ and $\varphi$ shows that $\varphi_{0}$ commutes with $h$. Lemma \ref{L:comp-commmu} tells us $h(k)$ must be constant. Since $\varphi$ is a path sequence, $h(k)$ must be $1$ and $\varphi(k) = k^{\mu}$ with $\mu \in \mathbb{Z}^{+}$.
\end{proof}

A simple result for integrable functions is
\begin{theorem}\label{T:all f}
	Given an arc $\mathcal{A}$ and an interlocked family $\Phi$, let
\[
	 \mathcal{I}_{\mathcal{A},\Phi}' = \{\,f \mid f \colon \mathcal{A} \to \mathbb{C}_{p},\  \text{and} \ \int_{\mathcal{A},\Phi}f(x)\,dx \quad exists \}.
\]
Then, $\mathcal{I}_{\mathcal{A},\Phi}'$ is a vector space and $f \to \int_{\mathcal{A},\Phi}f(x)\,dx$ is linear.
\end{theorem}
\begin{proof}
Scalar multiplication is obvious. For addition, suppose $f_{1}$, $f_{2} \in \mathcal{I}_{\mathcal{A},\Phi}'$.
Given $\varepsilon > 0$, there are $\varphi_{i}$ and $k_{i}$, $i = 1$,$2$, such that
\[
	\left| \int_{\mathcal{A},\Phi}f_{i}(x)\,dx - A_{f_{i},\varphi_{i}}(k) \right|_{p} < \varepsilon \qquad \text{for} \ k > k_{i}.
\]
By definition of \emph{interlocked}, there is a $\varphi \in \Phi$ and a $k_{0}$ such that $\varphi(k)$ is a common subsequence of $\varphi_{1}(k)$ and $\varphi_{2}(k)$ for $k > k_{0}$. Hence there's a $k_{0}'$ such that 
\[	
	\left| \int_{\mathcal{A},\Phi}f_{i}(x)\,dx - A_{f_{i},\varphi(k)} \right|_{p} < \varepsilon \qquad \text{for} \ k > k_{0}',\ i = 1, 2.
\]
Now, using 
\[
	A_{f_{i} + f_{2},\varphi(k)} = A_{f_{1},\varphi(k)} + A_{f_{2},\varphi(k)},
\]
we conclude
\[
	\left| \int_{\mathcal{A},\Phi}f_{1}(x)\,dx + \int_{\mathcal{A},\Phi}f_{2}(x)\,dx - A_{f_{1}+f_{2},\varphi(k)} \right|_{p} < \varepsilon \qquad \text{for} \ k >k_{0}'.
\]
This establishes that $f_{1} + f_{2} \in \mathcal{I}_{\mathcal{A},\Phi}^{'}$ and that
\[
	\int_{\mathcal{A},\Phi}(f_{1} + f_{2})(x)\,dx = \int_{\mathcal{A},\Phi}f_{1}(x)\,dx +\int_{\mathcal{A},\Phi}f_{2}(x)\,dx.
\]
That
\[
	\int_{\mathcal{A},\Phi}af_{1}(x)\,dx = a\int_{\mathcal{A},\Phi}f_{1}(x)\,dx
\]
is clear.
\end{proof} 
\section{General results for holomorphic functions}
In order to obtain further results on line integrals, we need to put conditions on $f(x)$.
Let $H(\mathcal{A})$ be the space of holomorphic functions from $\mathcal{A}$ into $\mathbb{C}_{p}$. Let $\Phi$ be an interlocked family of path sequences and
\begin{equation}\label{E:Iphi_alpha}
	\mathcal{I}_{\mathcal{A},\Phi} = \mathcal{I}_{\mathcal{A},\Phi}' \cap H(\mathcal{A}).
\end{equation}
	
To begin, equation \eqref{E:A_k} in Definition \ref{D:A_k} will be reshaped to take advantage of the fact that $f \in H(\mathcal{A})$.
\begin{theorem}\label{T:1st A_k c_n}
If $f \in H(\mathcal{A}(a,b))$, $\varphi$ is any path sequence and
\[
	f(x) = \sum_{n \geq 0}c_{n}(x - b)^{n},
\]
then
\begin{equation}\label{E:1st A_k c_n}
	A_{f,\varphi}(k) = - \sum_{n \geq p^{\varphi(k)} - 1}c_{n}(a - b)^{n+1} \sum_{j \geq 1} (-1)^{j - 1}\binom{n}{jp^{\varphi(k)} - 1},
\end{equation}
where $(-1)^{j - 1}$ is replaced by $-1$ if $p = 2$.
\end{theorem}
Note that the sum over $j$ has only a finite number of non-zero terms.
\begin{proof}
	Begin with Definition \ref{D:A_k}, equation \eqref{E:A_k},
\[
	 A_{f,\varphi}(k)= p^{-\varphi(k)}\sum_{x}(x-a)f(x),
\]
where $x=a+(b-a)\zeta$ and $\zeta$ runs through the $p^{\varphi(k)}$-th roots of unity.
Substitute into the sum using $x=a+(b-a)\zeta = b + (a - b)(1 - \zeta)$ and replace $\zeta$ by $\zeta_{k}^{r},\ r~=~1$, $2$, $\dots$, $p^{\varphi(k)}$. $\zeta_{k}$ is a fixed primitive $p^{\varphi(k)}$-th root of unity. Expand and sum the geometric series in $\zeta_{k}$ to finish.
\end{proof}
Theorem \ref{T:1st A_k c_n} immediately provides an important result:
\begin{theorem}\label{T:max int}

	If $f \in \mathcal{I}_{\mathcal{A},\Phi}$ for some interlocked family $\Phi$ and $\left|f(x)\right|_{p} \leq M$ on $\mathcal{A}$, then
\[
	\left| \int_{\mathcal{A},\Phi}f(x)\,dx  \right|_{p} \leq M \left| b - a \right|_{p} = MR.
\]	
\end{theorem}
\begin{proof}
Let
\[
	f(x) = \sum_{n \geq 0}c_{n}(x - b)^{n}.
\]
 We know, see \cite{wS1984}, that 
\[
 	\sup_{n \geq 0}{\left|c_{n}\right|_{p}R^{n}} = \sup_{x \in \mathcal{A}}{\left|f(x)\right|_{p}}.
\]
Hence, by Theorem \ref{T:1st A_k c_n}, for all $\varphi$ and any $k$,
\[ 
	\left| A_{f,\varphi(k)}(k) \right|_{p} \leq \sup_{n \geq 0}{\left|c_{n}\right|_{p}R^{n + 1}} \leq MR.
\]
 Let
\[
 	L = \int_{\mathcal{A},\Phi}f(x)\,dx.
\]
If $M = 0$, there is nothing to prove. If $M \neq 0$, we can use Definition \ref{D:integral} of the integral with $\varepsilon =  MR$ to conclude there is a $\varphi_{\varepsilon} \in \Phi$ and a $K$ such that 
 \[
 	 \left| A_{f,\varphi_{\varepsilon}}(K) - L \right|_{p} < MR.
 \]
Since  $\left| A_{f,\varphi_{\varepsilon}}(K) \right|_{p} \leq  MR$, we must have
 \[
 	\left| L \right|_{p} \ \leq \  MR.
 \]
\end{proof}

$\mathcal{I}_{\mathcal{A},\Phi}$ is a vector space. The following result shows $\mathcal{I}_{\mathcal{A},\Phi}$ is closed under uniform convergence.
\begin{theorem}\label{T:unif conv}
	Suppose $f_{i}(x) \in \mathcal{I}_{\mathcal{A},\Phi}$ for $i = 1$, $2$, $\dots$ and 
\[	
	f_{i}(x) \xrightarrow{\text{unif}} f(x).
\]
Then 
$f(x) \in \mathcal{I}_{\mathcal{A},\Phi}$ and
\[
	\int_{\mathcal{A},\Phi}f(x)\,dx = lim_{i \to \infty} \int_{\mathcal{A},\Phi}f_{i}(x)\,dx.
\]
\end{theorem}
\begin{proof}
	For any $\varepsilon > 0$ we have  $\left| f_{i + 1}(x) - f_{i}(x) \right|_{p} < \varepsilon$ for all $x$ and $i$ large. Hence
\[
	\left| \int_{\mathcal{A},\Phi}f_{i + 1}(x)\,dx - \int_{\mathcal{A},\Phi}f_{i}(x)\,dx \right|_{p} = \left| \int_{\mathcal{A},\Phi}f_{i + 1}(x) - f_{i}(x)\,dx \right|_{p} \leq \varepsilon R.
\] 
when $i$ is large and
\[
	lim_{i \to \infty} \int_{\mathcal{A},\Phi}f_{i}(x)\,dx = L \qquad \text{exists}.
\]
Now we need to show that 
\[
	\int_{\mathcal{A},\Phi}f(x)\,dx = L.
\]
Consider the identity, for any $\varphi$,
\begin{multline*}
\left|A_{f,\varphi}(k) - L \right|_{p}\\
= \left|A_{f,\varphi}(k) - A_{f_{i},\varphi}(k) + A_{f_{i},\varphi}(k) -  \int_{\mathcal{A},\Phi}f_{i}(x)\,dx  + \int_{\mathcal{A},\Phi}f_{i}(x)\,dx- L \right|_{p}.
\end{multline*}
The right side is made up of three easy to handle pieces. Given $\varepsilon$,
we can choose an $I$ large enough so that, regardless of $k$ or $\varphi$ (to be chosen),
\[
	\left|A_{f,\varphi}(k) - A_{f_{I},\varphi}(k) \right|_{p} < \varepsilon, \quad \text{and}  \quad \left| \int_{\mathcal{A},\Phi}f_{I}(x)\,dx- L \right|_{p} < \varepsilon.
\]
Having chosen $I$, we can now choose $\varphi \in \Phi$ so that
\[
	\left| A_{f_{I},\varphi}(k) -  \int_{\mathcal{A},\Phi}f_{I}(x)\,dx \right|_{p} <  \varepsilon
\]
for all sufficiently large $k$.
Putting these pieces together yields
\[
	\left|A_{f,\varphi}(k) - L \right|_{p} < \varepsilon \qquad\text{for all sufficiently large}\  k.
\]
Hence, by definition of an integral,
\[
	\int_{\mathcal{A},\Phi}f(x)\,dx = L.
\]
\end{proof}
 
For the next few theorems, the results can change depending on whether f(x) is bounded on $\mathcal{A}$ or not. The unbounded functions are more difficult to handle, so we need a new concept here.
\begin{definition}\label{D:contr coef}
We will say that $f \in H(\mathcal{A})$ has \emph{controlled coefficients of order $\beta$}, $\beta \geq 0$, if, given a power series for $f(x)$,
\[
	f(x) = \sum_{n \geq 0}c_{n}(x - b)^{n}
\]
there are real constants $M$ and $n_{o} > 0$, such that, with $R = |a - b|_{p}$, 
\[
	\left|c_{n}\right|_{p}R^{n} < M n^{\beta} \qquad \text{for} \ n > n_{0}.
\]
\end{definition}

The next theorem shows that this definition is independent of the choice of $b \in \mathcal{A}$.  Bounded functions are those with $\beta = 0$.

\begin{theorem}
	Suppose $b$, $b' \in \mathcal{A}$, $f \in H(\mathcal{A})$ and
\[
	f(x) = \sum_{n \geq 0}c_{n}(x - b)^{n} = \sum_{n \geq 0}c_{n}'(x - b')^{n}.
\]
Suppose there are real constants $M$ and $n_{o} > 0$, and $\beta \geq 0$ such that
\[
	\left|c_{n}\right|_{p}R^{n} < M n^{\beta} \qquad \text{for} \ n > n_{0}.
\]
Then there is an $ n_{0}'$ such that
\[
	\left|c_{n}'\right|_{p}R^{n} < M n^{\beta} \qquad \text{for} \ n > n_{0}'.
\]
\end{theorem}
\begin{proof}
Write 
\[
	f(x) = \sum_{n \geq 0}c_{n}(x - b' + b' - b)^{n}.
	\]
Expanding, rearranging terms and letting $b' - b = t(a - b)$, $|t |_{p} < 1$, yields
\[
	c_{j}'(a - b)^{j} = \sum_{n \geq j}c_{n}\binom{n}{j}(a - b)^{n }t^{n - j}.
\]
Then, if $j > n_{0}$,
\[	
	\left|c_{j}'\right|_{p}R^{j} \  \leq \  \max_{n \geq j}\left|c_{n}\right|_{p}R^{n}\left|t\right|_{p}^{n - j} \
\]
\[
	< \  \max_{n \geq j}Mn^{\beta}\left|t \right|_{p}^{n - j}\  = \ Mj^{\beta}\max_{n \geq j}\left(\frac{n}{j}\right)^{\beta}\left|t\right|_{p}^{n - j}.
\]
The last expression, $\left(\frac{n}{j}\right)^{\beta}\left|t\right|_{p}^{n - j}$, is a decreasing function of $n$ if $n > -\beta/\ln{\left|t\right|_{p}}$. If $j_{0}$ is chosen so $j_{0} > \max(n_{0}$,$-\beta/\ln{\left|t\right|_{p}})$ and $j \geq j_{0}$, then the maximum value is at $n = j$ and that value is 1.
Thus
\[
	\left|c_{j}'\right|_{p}R^{j} < Mj^{\beta} \qquad \text{for} \ j \geq j_{0}.
\]
\end{proof}
The next goal is to find a versions of Theorem \ref{T:1st A_k c_n} that will take advantage of the concept of \emph{controlled coefficients}.

The first step in proving Theorem \ref{T:1st A_k c_n} from the definition of $A_{\varphi,f}(k)$, (\ref{D:A_k}), was to substitute for $x$. This yielded
\[
	A_{\varphi,f}(k) = -\sum_{n \geq 0} p^{-\varphi(k)}c_{n}(a - b)^{n + 1}\sum_{\zeta}\zeta(1 - \zeta)^{n},
\]
where $\zeta$ runs through the $p^{\varphi(k)}$ -th roots of unity.
Let
\[
	 a_{k,n} = p^{-\varphi(k)}c_{n}(a - b)^{n + 1}\sum_{\zeta}\zeta(1 - \zeta)^{n}.
\]
\begin{lemma}\label{L: a_{k,n}}
	If $f \in H(\mathcal{A})$ and $f$ has controlled coefficients of order $\beta$, then
\[
	\nu_{p}(a_{k,n}) \geq -\varphi(k) - \log_{p}(MR) - \beta \log_{p}(n) + \frac{np}{p^{\varphi(k)}(p - 1)}.
\]
Furthermore, the expression on the right side of the inequality is an increasing function of $n$ if $n > \beta (p - 1)p^{\varphi(k) - 1}\log_{p}{e}$.
\end{lemma}
\begin{proof}
	Use
\[
	\nu_{p}(x) = -\log_{p}(|x|_{p})\ \text{and}\ \nu_{p}(1 - \zeta) \geq \frac{p}{p^{\varphi(k)}(p - 1)}.
\]
\end{proof}

Lemma \ref{L: a_{k,n}} immediately yields a version of Theorem \ref{T:1st A_k c_n} in which the tail of the series has been cut off.
\begin{theorem}\label{T:n_k1}
	If $f \in H(\mathcal{A}),f$ has controlled coefficients of order $\beta$\\ and $n_{k} \geq (\beta~+~2)\left(\frac{p~-~1}{p}\right)\varphi(k)p^{\varphi(k)}$  then, for large $k$,
\[
	A_{f,\varphi}(k) = - \sum_{n = p^{\varphi(k)} - 1}^{n_{k} - 1}c_{n}(a - b)^{n+1} \sum_{j 	\geq 1}(-1)^{j - 1}\binom{n}{jp^{\varphi(k)} - 1} + \eta(k),
\]
with $\eta(k) \to 0$ as $k \to \infty$.
\end{theorem}

The next modification of Theorem \ref{T:1st A_k c_n}, with a restriction on $\beta$, will be to remove terms where $n$ is not p-adically close to $-1$. This will require help from the binomial coefficients, so we need 
\begin{lemma}\label{L:bin remove n}
	If $j$,$n$,$t$,$N \in \mathbb{Z}^{+}$ and $n + 1 \not\equiv 0\pmod{p^{t}}$,then
\[
	\nu_{p}\binom{n}{jp^{N} - 1} > N - t.
\]
\begin{proof}
	This follows immediately from
\[
		\binom{n}{jp^{N} - 1} = \binom{n + 1}{jp^{N}}\frac{jp^{N}}{n + 1}.
\]
\end{proof}
\end{lemma}
We'll also need help from a function that relates to path sequences. 
\begin{definition}\label{D:auxfunc}
A function $\theta(k)$ will be called an \emph{auxiliary function} if there is a $k_{0} \in \mathbb{Z}^{+}$ such that  $\theta(k)$ is defined for $k \in \mathbb{Z}^{+}$ with $k \geq k_{0}$, $\theta(k) \in \mathbb{Z}^{+}$ and
\[
	 \lim_{k \to \infty}\theta(k) = \infty.
\]
 
\end{definition}
\begin{theorem}\label{T:best A_f,phi1}
	Suppose $f \in H(\mathcal{A}),f$ has controlled coefficients of order $ \beta< 1$ and $n_{k} = \left(1 + \left[(\beta~+~2)\left(\frac{p~-~1}{p}\right)\varphi(k)\right]\right)p^{\varphi(k) }$.\\
	Furthermore, suppose $\theta$ is an auxiliary function that satisfies
\[
		(1 - \beta)\varphi(k) - \beta \log_{p} \varphi(k) - \theta(k) \to \infty \qquad \text{as} \quad k \to\infty.
\]
Then, for large $k$
\begin{equation}\label{E:best A_f,phi1}
	A_{f,\varphi}(k) = - \sum_{\substack{n = p^{\varphi(k)} - 1\\ n \equiv -1 \pmod{p^{\theta(k)}}}}^{n_{k} - 1}c_{n}(a - b)^{n+1} \sum_{j \geq 1}(-1)^{j - 1}\binom{n}{jp^{\varphi(k)} - 1} + \eta(k),
\end{equation}
with $\eta(k) \to 0$ as $k \to \infty$.
\end{theorem}
\begin{proof}
 This is a straightforward calculation of the size of the omitted terms using Theorem \ref{T:n_k1}, $\nu_{p}(x) = -\log_{p}(|x|_{p})$, $n < n_{k}$, $\beta < 1$, Lemma \ref{L:bin remove n} and the assumption on $\theta(k)$.
\end{proof}

In order to later state a sufficient condition for an integral to exist, slightly modified versions of the last two theorems will be needed. The proofs are similar.
\begin{theorem}\label{T:m_k}
	Suppose $f \in H(\mathcal{A})$, $f$ has controlled coefficients of order $\beta$ and $m_{k}~\geq~(\beta~+~2)\left(\frac{p~-~1}{p}\right)\varphi(k)p^{\varphi(k + 1)}$. Also suppose $\varphi(k)$ satisfies
\[
	\varphi(k) - (\beta + 1)\nabla\varphi(k) - \beta\log_{p}\varphi(k) \to \infty \qquad \text{as} \quad k \to\infty,
\]
with $\nabla\varphi(k) = \varphi(k + 1) - \varphi(k)$.
Then, for large $k$
\[
	A_{f,\varphi}(k + 1) = - \sum_{n = p^{\varphi(k + 1)} - 1}^{m_{k} - 1}c_{n}(a - b)^{n+1} \sum_{j \geq 1}(-1)^{j - 1}\binom{n}{jp^{\varphi(k + 1)} - 1} + \eta(k),
\]
with $\eta(k) \to 0$ as $k \to \infty$.
\end{theorem}
\begin{theorem}\label{T:best A_f,phi2}
	Suppose $f \in H(\mathcal{A}),f$ has controlled coefficients of order $ \beta< 1$, $m_{k} = \left(1 + \left[(\beta~+~2)\left(\frac{p~-~1}{p}\right)\varphi(k)\right]\right)p^{\varphi(k + 1) }$ and $\varphi(k)$ satisfies
\[
	\varphi(k) - (\beta + 1)\nabla\varphi(k) - \beta\log_{p}\varphi(k) \to \infty \qquad \text{as} \quad k \to\infty.
\]
Furthermore, suppose $\theta_{2}$ is an auxiliary function that satisfies
\[
		(1 - \beta)\varphi(k + 1) - \beta \log_{p} \varphi(k) - \theta_{2}(k) \to \infty \qquad \text{as} \quad k \to\infty.
\]
Then, for large $k$
\begin{equation}\label{E:best A_f,phi2}
	A_{f,\varphi}(k + 1) = - \sum_{\substack{n = p^{\varphi(k + 1)} - 1\\ n \equiv -1 \pmod{p^{\theta_{2}(k)}}}}^{m_{k} - 1}c_{n}(a - b)^{n+1} \sum_{j \geq 1}(-1)^{j - 1}\binom{n}{jp^{\varphi(k + 1)} - 1} + \eta(k),
\end{equation}
with $\eta(k) \to 0$ as $k \to \infty$.
\end{theorem}

Now we're ready to look at changing $a$ and $b$ within the definition of an integral. Since any point in the interior of a circle can be used as its center and any point in an open disc can be used as its center, we would like the value of an integral to be invariant under these changes. If $f$ is bounded this will be true. If $f$ is unbounded, the situation is more complicated.
\begin{theorem}\label{T:change a,b}
	If $f \in \mathcal{I}_{\mathcal{A}(a,b),\Phi}$, $f$ is bounded on $\mathcal{A}$, $|a - a'|_{p} < R$ and $|b - b'|_{p} < R$, then
\[
	\int_{\mathcal{A}(a,b),\Phi}f(x)\,dx = \int_{\mathcal{A}(a',b'),\Phi}f(x)\,dx.
\]
\end{theorem}
\begin{theorem}
If $f$ is unbounded on $\mathcal{A}(a,b)$, $f$ has controlled coefficients of order $\beta< 1/2$ and $|a - a'|_{p} < Rp^{\frac{-\beta}{1 - \beta}}$ and $|b - b'|_{p} <  Rp^{\frac{-\beta}{1 - \beta}}$, then
\[
	\int_{\mathcal{A}(a,b),\Phi}f(x)\,dx = \int_{\mathcal{A}(a',b'),\Phi}f(x)\,dx.
\]
\end{theorem}
\begin{proof}
Let
\[
	f(x) = \ \sum_{n \geq 0}c_{n}(x - b)^{n}.
\]
	It's simplest to look at changing $a$ and $b$ separately. \\Let's assume $|a - a'|_{p}~<~Rp^{\frac{-\beta}{1 - \beta}}$. This will cover both bounded and unbounded $f$.
	It will be helpful to have this lemma.
\begin{lemma}
Let $a - a' = t(a - b)$ with $|t|_{p} < 1$. Suppose $n \equiv -1 \pmod{p^{\theta(k)}}$.
Then,
\[
	\left(\frac{a' - b}{a - b}\right)^{n + 1} = 1 + s,\qquad \text{with}\ |s|_{p} < (\max (|t|_{p},|p|_{p})^{\theta(k)}.
\]
\begin{proof}
Let $n = dp^{\theta(k)} - 1$. 
\[
	\left(\frac{a' - b}{a - b}\right)^{n + 1} = \left(1 + \frac{a' -a}{a - b}\right)^{dp^{\theta(k)}} = \left(1 + t'\right)^{p^{\theta(k)}}, \quad \text{with}\ |t'|_{p} \leq |t|_{p}.
\]
The lemma is then the conclusion of Robert's \emph{Fundamental Inequality: Second form}, \cite{aR2000}.
\end{proof}
\end{lemma}
For clarity, we'll write $A_{f,\varphi,a}(k)$ and $A_{f,\varphi,a'}(k)$. For simplicity, we'll write
\[
	S(\varphi,k,n) \overset{\textup{def}}{=} \sum_{j = 1}^{\left[(n + 1)/p^{\varphi(k)}\right]} (-1)^{j - 1}\binom{n}{jp^{\varphi(k)} - 1}.
\]
We have, for any $\varphi$, by Theorem \ref{T:best A_f,phi1},
\[
	A_{f,\varphi,a}(k) =  - \sum_{\substack{n = p^{\varphi(k)} - 1\\ n \equiv -1 \pmod{p^{\theta(k)}}}}^{n_{k} - 1}c_{n}(a - b)^{n+1}S(\varphi,k,n) + \eta_{1}(k)
\]
and
\begin{multline*}
	A_{f,\varphi,a'}(k) =  - \sum_{\substack{n = p^{\varphi(k)} - 1\\ n \equiv -1 \pmod{p^{\theta(k)}}}}^{n_{k} - 1}c_{n}(a' - b)^{n+1}S(\varphi,k,n) + \eta_{2}(k)\\
= - \sum_{\substack{n = p^{\varphi(k)} - 1\\ n \equiv -1 \pmod{p^{\theta(k)}}}}^{n_{k} - 1}c_{n}(a - b)^{n+1}\left(\frac{a' - b}{a - b}\right)^{n + 1}S(\varphi,k,n) + \eta_{2}(k).
\end{multline*}
with $\eta_{1}(k)$ and $\eta_{2}(k) \to 0$ as $k \to \infty$.
Hence,
\[
	\left| A_{f,\varphi,a'}(k) - A_{f,\varphi,a}(k)\right|_{p} \leq \max_{n} \left| c_{n}(a - b)^{n+1}\right|_{p}\left|\left(\frac{a' - b}{a - b}\right)^{n + 1} - 1\right|_{p} + |\eta(k)|_{p}.
\]
with $|\eta(k)|_{p} \to 0$ as $k \to \infty$.
We're now ready to apply the definition of integral.
Let 
\[
	L = \int_{\mathcal{A}(a,b),\Phi}f(x)\,dx.
\]
Given $\varepsilon > 0$, let $\varphi \in \Phi$ be such that $|A_{f,\varphi,a}(k)) - L|_{p} < \varepsilon$ for large $k$.\\

$\textbf{Case 1}$: $|t|_{p} \leq |p|_{p}$. Let $\theta(k) = 1 + \left[\frac{1}{2}\varphi(k)\right]$. Then for sufficiently large $k$ we have
\[
	\left| A_{f,\varphi,a'}(k) - A_{f,\varphi,a}(k)\right|_{p} \leq Mn_{k}^{\beta}p^{-\theta(k)} + |\eta(k)|_{p} < \  \text{const.}\cdot\varphi^{\beta}(k)p^{(\beta - \frac{1}{2})\varphi(k)} + |\eta(k)|_{p} < \varepsilon.
\]
This shows that $|A_{f,\varphi,a'}(k)) - L|_{p} < \varepsilon$ for large $k$ and
\[
	\int_{\mathcal{A}(a',b),\Phi}f(x)\,dx = \int_{\mathcal{A}(a,b),\Phi}f(x)\,dx
\]

$\textbf{Case 2}$: $|p|_{p}< |t|_{p} < p^{\frac{-\beta}{1 - \beta}}$. $\varphi$ is chosen as in Case 1, but $\theta$ is not. Let
\[
	\theta(k) = \left[\frac{1}{1 - \log_{p} |t|_{p}}\varphi(k)\right].
\]
$\theta(k)$ satisfies the condition in Theorem \ref{T:best A_f,phi1}.
Then
\[
	\left| A_{f,\varphi,a'}(k) - A_{f,\varphi,a}(k)\right|_{p} \leq Mn_{k}^{\beta}p^{-\theta(k)}  + |\eta(k)|_{p}< \  \text{const.}\cdot\varphi^{\beta}(k)p^{(\beta + \frac{\log |t|_{p}}{1 - \log |t|_{p}})\varphi(k)} + |\eta(k)|_{p} < \varepsilon.
\]
with $|\eta(k)|_{p} \to 0$ as $k \to \infty$.\\

The coefficient of $\varphi(k) $ in the exponent is negative, so again we have
\[|A_{f,\varphi,a'}(k)) - L|_{p} < \varepsilon,
\]
for large $k$, and thus
\[
	\int_{\mathcal{A}(a',b),\Phi}f(x)\,dx = \int_{\mathcal{A}(a,b),\Phi}f(x)\,dx.
\]

Now let's consider changing $b$ to $b'$.
This will be similar to changing $a$, but more complicated because the coefficients $c_{n}$ are changing.
For clarity, we will write $A_{f,\varphi,b}(k)$ and $A_{f,\varphi,b'}(k)$.
Let
\[
	b' - b = t(a - b),\  \textup{with}\  |t|_{p} < p^{\frac{-\beta}{1 - \beta}}.
\]
\[
	f(x) = \ \sum_{n \geq 0}c_{n}(x - b)^{n} \ = \ \sum_{j \geq 0}c_{j}'(x - b')^{j},
\]
with
\[
	c_{j}' = \sum_{n \geq j}c_{n}\binom{n}{j}(a - b)^{n - j}t^{n - j}.
\]
The variables can be relabeled to obtain
\begin{equation}\label{E:changeb}
	c_{n}' = c_{n} + \sum_{m \geq 1}c_{m + n}\binom{m + n}{n}(a - b)^{m}t^{m}.
\end{equation}
With $\theta(k)$ yet to be chosen and
\[
	S(\varphi,k,n) = \sum_{j = 1}^{\left[(n + 1)/p^{\varphi(k)}\right]} (-1)^{j - 1}\binom{n}{jp^{\varphi(k)} - 1},
\]
 Theorem \ref{T:best A_f,phi1} gives us
\[
	A_{f,\varphi,b}(k) =  - \sum_{\substack{n = p^{\varphi(k)} - 1\\ n \equiv -1 \pmod{p^{\theta(k)}}}}^{n_{k} - 1}c_{n}(a - b)^{n+1}S(\varphi,k,n) + \eta_{1}(k),
\]
and 
\[
	A_{f,\varphi,b'}(k) =  - \sum_{\substack{n = p^{\varphi(k)} - 1\\ n \equiv -1 \pmod{p^{\theta(k)}}}}^{n_{k} - 1}c_{n}'(a - b')^{n+1}S(\varphi,k,n) + \eta_{2}(k).
\]
with $\eta_{1}(k)$ and $\eta_{2}(k) \to 0$ as $k \to \infty$.

Substituting for $c_{n}'$ and subtracting $A_{f,\varphi,b'}(k)$ we get
\[
	A_{f,\varphi,b}(k) - A_{f,\varphi,b'}(k) = \sum_{\substack{n = p^{\varphi(k)} - 1\\ n \equiv -1 \pmod{p^{\theta(k)}}}}^{n_{k} - 1}\left(\left(\frac{a - b'}{a - b}\right)^{n + 1} - 1\right) c_{n}(a - b)^{n+1}S(\varphi,k,n)
\]
\[
	+  \sum_{\substack{n = p^{\varphi(k)} - 1\\ n \equiv -1 \pmod{p^{\theta(k)}}}}^{n_{k} - 1}\left(\frac{a - b'}{a - b}\right)^{n + 1}\sum_{m \geq 1}c_{m + n}(a - b)^{m + n+1}\binom{m + n}{n}t^{m}S(\varphi,k,n) + \eta(k),
\]
where $\eta(k) \to 0$ as $k \to \infty$.\\
Since we can write $a - b' = a' - b$ with $a' = a + (b - b')$, the considerations of how we can change $a$ show that, if we use the same $\theta(k)$, the first large summand can be made $< \varepsilon$ for k sufficiently large.
Now we need to deal with the second large summand. The difficulty is that while $n$ runs through a finite set of values, $m$ goes from $1$ to $\infty$.

We have $\left|\left(\frac{a - b'}{a - b}\right)^{n + 1}\right|_{p} \leq 1$ and $|S(\varphi,k,n)|_{p} \leq 1$, so we will just consider the expressions
\[
	c_{m + n}(a - b)^{m + n+1}\binom{m + n}{n}t^{m},
\]
with $p^{\varphi(k)} - 1 \leq n < n_{k}$ and $m \in \mathbb{Z}^{+}$.
The factor $t^{m}$ will be useful when $m$ is large, but will not help for small $m$.
As with $a$, we will consider two cases, depending on the size of $|t|_{p}$.\\

\noindent$\textbf{Case 1}$: $|t|_{p} \leq |p |_{p}$. Let $\theta(k) = \left[\frac{1}{2}\varphi(k)\right]$. First we'll consider $m > \theta(k)$.
\[
	\left|c_{m + n}(a - b)^{m + n+1}\binom{m + n}{n}t^{m}\right|_{p} < MR(m + n_{k})^{\beta}p^{-m}.
\]
This last expression is a decreasing function of $m$ for $m \geq 1$, so, since $m > \theta(k)$, we have
\[
	MR(m + n_{k})^{\beta}p^{-m} < MR\left(n_{k} + \frac{1}{2}\varphi(k)\right)^{\beta}p^{-\frac{\varphi(k)}{2}}.
\]
The subsitution of $n_{k} = \left(1 + \left[(\beta~+~2)\left(\frac{p~-~1}{p}\right)\varphi(k)\right]\right)p^{\varphi(k)}$ leads to
\[
	\left|c_{m + n}(a - b)^{m + n+1}\binom{m + n}{n}t^{m}\right|_{p} < \ \text{const.}\cdot\varphi^{\beta}(k)p^{(\beta - \frac{1}{2})\varphi(k)}.
\]
Now let's consider $m \leq \theta(k)$.
\[
\left|c_{m + n}(a - b)^{m + n+1}\binom{m + n}{n}t^{m}\right|_{p} < \ MR(\theta(k) + n_{k})^{\beta}\left|\binom{m + n}{n}\right|_{p}.
\]
Since $\nu_{p}(m) \leq \log_{p} \theta(k)$, Lemma \ref{L:bin remove n} gives us
\[
	\left|\binom{m + n}{n}\right|_{p} \ \leq \ \theta(k)p^{-\theta(k)}\  = \ \frac{\varphi(k)}{2}p^{\frac{-\varphi(k)}{2}}.
\]
Now we have
\[
	\left|c_{m + n}(a - b)^{m + n+1}\binom{m + n}{n}t^{m}\right|_{p} < MR\left(n_{k} + \frac{1}{2}\varphi(k)\right)^{\beta}\frac{\varphi(k)}{2}p^{-\frac{-\varphi(k)}{2}}.
\]
As earlier in $\textbf{Case 1}$, this leads to an expression dominated by $p^{(\beta - \frac{1}{2})\varphi(k)}$.

Since $\beta < \frac{1}{2}$ we can conclude, using the same argument used for changing $a$, that if $|t|_{p} \leq |p|_{p}$ then
\[
	\int_{\mathcal{A}(a,b'),\Phi}f(x)\,dx = \int_{\mathcal{A}(a,b),\Phi}f(x)\,dx.
\]
$\textbf{Case 2}$: $|p|_{p} < |t|_{p} < p^{\frac{-\beta}{1 - \beta}}$. Here, we must use the same $\theta(k) = \left[\frac{1}{1 - \log_{p} |t|_{p}}\varphi(k)\right]$ as before, and also consider $m > \theta(k)$ separately from $m \leq \theta(k)$. The results will look much like they did in Case 2 for changing $a$. This completes the proof of Theorem\ \ref{T:change a,b}.
\end{proof}

\section{Calculation of Integrals}\label{S:calc ints}

\begin{theorem}\label{T:holo=0}
If $f(x)$ is holomorphic on the closed disc $D^{+}(a,R)$ and $\left| a - b\right|_{p} = R$, then for any $\Phi$ we have
\[
	\int_{\mathcal{A}(a,b),\Phi}f(x)\,dx = 0.
\]
\end{theorem}
\begin{proof}
By Theorem \ref{T:varphi=kPhi} it is sufficient to prove this theorem for the case $\varphi(k) = k$.
Let $\varepsilon > 0$.

Since $\lim_{n \to \infty}\left|c_{n}\right|_{p}R^{n} = 0$, there is an $N_{\varepsilon}$ such that $\left|c_{n}\right|_{p}R^{n + 1} < \varepsilon$ if $n > N_{\varepsilon}$.
If $K_{\varepsilon}$ is chosen so $p^{K_{\varepsilon}} - 1 > N_{\varepsilon}$, Theorem \ref{T:1st A_k c_n} shows 
\[
	\left| A_{f,\varphi}(k) - 0 \right|_{p} < \varepsilon, \qquad \text{when}\ k > K_{\varepsilon}
\]
and, hence,
\[
	\int_{\mathcal{A},\Phi}f(x)\,dx = 0.
\]
\end{proof}
\begin{theorem}\label{T:intf'=0}
Let $f \in H(\mathcal{A})$ and suppose $f$ has controlled coefficients of order $\beta <  \frac{1}{2}$, then, for any $\Phi$,
\[
	\int_{\mathcal{A},\Phi}f'(x)\,dx = 0.
\]
\end{theorem}
\begin{proof}
By Theorem \ref{T:varphi=kPhi} it is sufficient to prove this theorem for the case $\varphi(k) = k$. 
	Let 
\[
		f(x) = \sum_{n \geq 0}c_{n}(x - b)^{n}\quad \text{and}\quad \theta(k) = 1 + [k/2].
\]
\[
f'(x) = \sum_{n \geq 0}(n + 1)c_{n + 1}(x - b)^{n}.
\]
By Theorem \ref{T:best A_f,phi1} we have
\[
	A_{f',\varphi}(k) = - \sum_{\substack{n = p^{k} - 1\\ n \equiv -1 \pmod{p^{\theta(k)}}}}^{n_{k} - 1}(n + 1)c_{n + 1}(a - b)^{n+1} \sum_{j = 1}^{n'} (-1)^{j - 1}\binom{n}{jp^{k} - 1} + \eta(k),
\]
where $\eta(k) \to 0$ as $k \to \infty$.
Hence,
\[
	\left|A_{f',\varphi}(k)\right|_{p} \leq \max_{n < n_{k}}\left|(n + 1)c_{n + 1}(a - b)^{n+1}\right|_{p} + |\eta(k)|_{p}.
\]
We have $|n + 1|_{p} \leq p^{-\theta(k)}$ and 
\[
	|c_{n + 1}(a - b)^{n+1}|_{p} = |c_{n + 1}|_{p}R^{n + 1} \leq M(n + 1)^{\beta} \leq  M(n_{k} )^{\beta}.
\]
Applying $n_{k} = \left(1 + \left[(\beta~+~2)\left(\frac{p~-~1}{p}\right)k\right]\right)p^{k}$ we obtain
\[
	\left|A_{f',\varphi}(k)\right|_{p}\  \leq \ \text{const.}\cdot k^{\beta}p^{(\beta - \frac{1}{2})k} + |\eta(k)|_{p}.
\]
Given any $\varepsilon > 0$, if $k$ is sufficiently large we have
\[
	\left|A_{f',\varphi}(k)\right|_{p}\ < \varepsilon,
\]
and therefore,
\[
	\int_{\mathcal{A},\Phi}f'(x)\,dx = 0. 
\]
\end{proof}
Here are some applications of Theorem \ref{T:intf'=0}.
\begin{theorem}
\leavevmode
\begin{enumerate}
\item
If $|x_{0} -a|_{p} \leq |b - a|_{p}$ and $x_{0} \notin \mathcal{A}(a,b)$ then
\[
	\int_{\mathcal{A}(a,b),\Phi}(x - x_{0})^{-m}\,dx = 0,\qquad \text{for} \  m \in \mathbb{Z} \ \text{and}\ m \neq -1;
	\]
\item let $\pi \in \mathbb{C}_{p}$ with $\nu_{p}(\pi) = \frac{1}{p - 1}$, then 
\[
	\int_{\mathcal{A}(1,0),\Phi}exp(\pi x)\,dx = 0;
\]
\item if $t \in \mathbb{Z}_{p}$ and $t \neq -1$, then
\[
	\int_{\mathcal{A}(1,0),\Phi}(1 + x)^{t}\,dx = 0;
\]
\end{enumerate}
\end{theorem}
\begin{proof}
The integrands are derivatives of bounded functions on their respective arcs.
\end{proof}
\begin{theorem}[Integration by Parts]\label{T:intbyparts}
If $f$, $g \in H(\mathcal{A})$, $f(x)g(x)$ has controlled coefficients of order $ < \frac{1}{2}$ and $f(x)g'(x) \in \mathcal{I}_{\mathcal{A},\Phi}$, then $f'(x)g(x) \in \mathcal{I}_{\mathcal{A},\Phi}$ and
\[
	\int_{\mathcal{A},\Phi}f'(x)g(x)\,dx = -\int_{\mathcal{A},\Phi}f(x)g'(x)\,dx
\]
\end{theorem}
\begin{proof}
Apply Theorem \ref{T:intf'=0} to the product $fg$.
\end{proof}
\textbf{Rational Functions}
To be able to integrate rational functions in $H(\mathcal{A})$, it is enough to be able to integrate functions of the form $(x - x_{0})^{m}$ with $m \in \mathbb{Z}$. Previous theorems have taken care of all cases except functions of the form $(x - x_{0})^{-1}$ with $x_{0} \notin \mathcal{A}$ and $|x_{0} - a|_{p} \leq |a - b|_{p}$.

\begin{theorem}\label{T:ratfunc int}
Suppose $x_{0} \notin \mathcal{A}$ and $|x_{0} - a|_{p} \leq |a - b|_{p} = R$. Then
\begin{enumerate}
\item If $|x_{0} - a|_{p} < R$, we have, for any $\Phi$,
\[
	\int_{\mathcal{A},\Phi}\frac{1}{x - x_{0}}\,dx = 1;
\]
\item If $|x_{0} - a|_{p} = R$ and $x_{0} \notin \mathcal{A}$, there is a $\lambda_{0}$ such that for any $\alpha \in \mathbb{Z}_{\geq 0}$
\[
\int_{\mathcal{A},\alpha + \lambda_{0}k}\frac{1}{x - x_{0}}\,dx \  = \  \frac{1}{1 - \omega^{p^{\alpha}}\left(\frac{a - x_{0}}{a - b}\right)},
\]
where $\omega$ is the Teichm\"{u}ller character.
\end{enumerate}
\end{theorem}

In the complex case this integral around a circle can be $\pm 2\pi i$ and, given an orientation for the circle, the integral is constant for all $x_{0}$ in the interior of the circle. In $\mathbb{C}_{p}$, after choosing an $\alpha$, the integral is constant for all $x_{0}$ within an open disc of radius $R$.

\begin{proof}
We'll work straight from the preliminary definition of an integral, equation~\eqref{E:1st int}. Substituting $x = a + (b - a)\zeta$ yields
\[
	\lim_{k \to \infty}p^{-\varphi(k)}\sum_{\zeta}\dfrac{(b - a)\zeta}{a + (b - a)\zeta - x_{0}}\  = \ \lim_{k \to \infty}\dfrac{1}{1 - \left(\dfrac{a - x_{0}}{a - b}\right)^{p^{\varphi(k)}}}.
\]
The summation was obtained by noting that if $x = \frac{\zeta}{\zeta - c}$, then $\zeta = \frac{cx}{x - 1}$ and, as $\zeta$ runs through the $p^{\varphi(k)}$-th roots of unity, $x$ runs through the roots of 
\[
	(cx)^{p^{\varphi(k)}} - (x - 1)^{p^{\varphi(k)}}.
\]
Here,
\[
	c = \frac{a - x_{0}}{a - b}.
\]
If $x_{0}$ is interior to the circle, i.e. $|a - x_{0}|_{p} < R$, then the limit is 1 and we have, for any $\varphi$,
\[
	\int_{\mathcal{A},\varphi}\frac{1}{x - x_{0}}\,dx \ = 1.
\]
In particular, the result is valid for $\varphi(k) = k$, and, by use of Theorem \ref{T:varphi=kPhi}, result (1) is proven.\\
If $|a - x_{0}|_{p} = R$, then 
\[
	\left|\left(\frac{a - x_{0}}{a - b}\right)\right|_{p} = 1.
\]
Let $v = \omega\left(\frac{a - x_{0}}{a - b}\right)$. We can write
\[
	\frac{a - x_{0}}{a - b}\  = \ vu, \qquad \text{with} \ |u - 1|_{p} < 1.
\]
Since $lim_{k \to \infty}u^{p^{\varphi(k)}} = 1$,
\begin{equation}\label{w,v}
	\lim_{k \to \infty}\left(\frac{a - x_{0}}{a - b}\right)^{{p^{\varphi(k)}}} = \lim_{k \to \infty}v^{{p^{\varphi(k)}}}.
\end{equation}
Let $n$ be the least positive integer such that $v^{n} = 1$. We have $(n,p) = 1$ and, because $x_{0} \notin \mathcal{A}$, $n \neq 1$. Let $\lambda_{0}$ be the least positive integer such that $p^{\lambda_{0}} \equiv 1 \pmod{n}$.

Now let $\varphi(k) = \alpha + \lambda_{0}k$. $\alpha \in \mathbb{Z}_{\geq 0}$,
\[
	v^{p^{\varphi(k)}} = \left(v^{p^{\lambda_{0}k}}\right)^{p^{\alpha}} = \ v^{p^{\alpha}}.
\]
Thus
\[
	\int_{\mathcal{A},\alpha + \lambda_{0}k}\frac{1}{x - x_{0}}\,dx \  = \  \frac{1}{1 - v^{p^{\alpha}}} =  \frac{1}{1 - \omega^{p^{\alpha}}\left(\frac{a - x_{0}}{a - b}\right)}.
\]
\end{proof}
The next result shows that the integral of $(x - x_{0})^{-1}$ has the same value as it has for some path sequence of the form $ \alpha + \lambda k$, $0 \leq \alpha < \lambda$, regardless of the path sequence $\varphi$ being used. (Assuming the integral exists for $\varphi$.)
\begin{theorem}
If
\[
	\int_{\mathcal{A},\varphi}\frac{1}{x - x_{0}}\,dx
\]
exists, then there are $\alpha$ and $\lambda \in \mathbb{Z}$, $0 \leq \alpha < \lambda$, such that $\varphi(k)$ is a subsequence of $\alpha + \lambda k$ for all $k$ sufficiently large and 
\[
	\int_{\mathcal{A},\varphi}\frac{1}{x - x_{0}}\,dx \ = \ \int_{\mathcal{A},\alpha + \lambda k}\frac{1}{x - x_{0}}\,dx.
\]
\end{theorem}
\begin{proof}
The expression $v^{p^{\varphi(k)}}$ in equation \eqref{w,v} takes on only finitely many different values, so if $\lim_{k \to \infty}v^{p^{\varphi(k)}}$ exists, then $v^{p^{\varphi(k)}}$ is constant for large values of $k$. Thus, for large $k$, we have
\[
	v^{p^{\varphi(k+ 1)}} = v^{p^{\varphi(k)}}, \qquad v^{p^{\varphi(k + 1)} - p^{\varphi(k)}} = 1
\]
and, using the definitions of $n$ and $\lambda$ from Theorem \ref{T:ratfunc int}
\[
	p^{\varphi(k + 1)} \equiv p^{\varphi(k)} \pmod{n}, \qquad p^{\varphi(k + 1) - p^{\varphi(k)}} \equiv 1 \pmod{n}.
\]
Hence $\varphi(k + 1) \equiv \varphi(k)\pmod{\lambda_{0}}$. Let $\alpha \equiv \varphi(k)\pmod{\lambda_{0}}$ when $k$ is large. We have shown that $\varphi(k)$, for large $k$, is a subsequence of $\alpha + \lambda_{0}k$. Therefore, by Theorem \ref{T:alpha1=alpha2} we can say 
\[
	\int_{\mathcal{A},\varphi(k)}\frac{1}{x - x_{0}}\,dx \ = \ \int_{\mathcal{A},\alpha + \lambda_{0}k}\frac{1}{x - x_{0}}\,dx.
\]
\end{proof}
We know that each $\mathcal{I}_{\mathcal{A},\Phi_{\alpha}}$, $\alpha \geq 0$, contains all rational functions without poles in $\mathcal{A}$. If $\alpha < 0$, Theorem \ref{T:alpha1=alpha2} and the Chinese Remaider Theorem let us see that $\Phi_{\alpha}$ allows all rational functions without poles in $\mathcal{A}$ to be integrated.\\

Therefore, because $\mathcal{I}_{\mathcal{A},\Phi_{\alpha}}$ is closed under uniform convergence, (Theorem \ref{T:unif conv}), we can conclude
\begin{theorem}\label{T:Krasner}
For each $\alpha \in \mathbb{Z}$ the set $\,\mathcal{I} _{\mathcal{A},\Phi_{\alpha}}$ contains all functions that are uniform limits of rational functions on $\mathcal{A}$, that is, all Krasner analytic functions on  $\mathcal{A}$.
\end{theorem}

\begin{theorem}[Z - P]\label{T:z-p}
Suppose $f(x)$ is meromorphic on $D = D^{+}(a,R)$ and that all zeroes and poles of $f$ can be found in the open discs $D_{i} = D^{-}(z_{i},R)$ for $i = 1$, $\dots$, $M$. Let $b$ be on the circle $|x - a|_{p} = R$, but not in any $D_{i}$. Let $Z_{i}$ be the number of zeros  and $P_{i}$ the number of poles of $f$ in $D_{i}$, counting multiplicity. Let $\alpha \in \mathbb{Z}_{\geq 0}$.
Then,
\[
	\int_{\mathcal{A}(a,b),\Phi_{\alpha}}\frac{f'(x)}{f(x)}\,dx = \sum_{i = 1}^{M}\frac{1}{1 - \omega^{p^{\alpha}}\left( \frac{a - z_{i}}{a - b}\right)}\left(Z_{i} - P_{i}\right).
\]	
\end{theorem}
\begin{proof}
Suppose $f$ has zeros of order $m_{i}$ at $x_{i}$ and poles of order $\overline{m}_{j}$ at $\overline{x}_{j}$. Then
\[
	\frac{f'(x)}{f(x)} = h(x) + \sum_{i}\frac{m_{i}}{x - x_{i}} - \sum_{j}
\frac{\overline{m}_{j}}{x - \overline{x}_{j}}.
\]
$h(x) \in H(D)$, so its integral is $0$, and if $x_{i} \in D_{i}$, then
\[
	\int_{\mathcal{A}(a,b),\Phi_{\alpha}}\frac{m_{i}}{x - x_{i}}\,dx = \frac{m_{i}}{1 - \omega^{p^{\alpha}}\left( \frac{a - x_{i}}{a - b}\right)} = \frac{m_{i}}{1 - \omega^{p^{\alpha}}\left( \frac{a - z_{i}}{a - b}\right)}.
\]
The theorem follows.
\end{proof}
\begin{corollary}\label{C:Z-P}
If a meromorphic function $f$ on $D^{+}(a,R)$ has all of its zeros and poles interior to the circle $|x - a|_{p} = R$ and $|a - b|_{p} = R$, then, with $\alpha \in \mathbb{Z}_{\geq 0}$,
\[
	\int_{\mathcal{A}(a,b),\Phi_{\alpha}}\frac{f'(x)}{f(x)}\,dx = Z - P.
\]
\end{corollary}
\begin{proof}
Let  $z_{1} = a$.
\end{proof}
Now let's look at the question   of integration by substitution. If we have an integral
\[
		\int_{\mathcal{A}(a,b),\Phi}f(x)\,dx,
\]
are there mappings $x \colon \mathcal{A}(a_{1},b_{1}) \to \mathcal{A}(a,b)$ with $x(a_{1}) = a$ and $x(b_{1}) = b$, such  that
\[
		\int_{\mathcal{A}(a,b),\Phi}f(x)\,dx = \int_{\mathcal{A}(a_{1},b_{1}),\Phi}f(x(t))x'(t)\,dt\ ?
\]
If $f$ is Krasner analytic on $\mathcal{A}(a,b)$, a uniform limit of rational functions having no poles in $\mathcal{A}(a,b)$, the answer is yes.
Before stating the substitution theorem, it will be useful to have the following result.
\begin{theorem}\label{T:substmap}
Suppose
\[
	 f \colon D^{+}(a_{1},R) \longmapsto D^{+}(a,r), \quad f(x)\in H(D^{+}(a_{1},R)) \quad \textup{and} \quad f(a_{1}) = a.
\]
Then $f$ is a one-to-one, onto mapping if and only if $f$ can be written
\[
	f(x) = a + \gamma (x - a_{1}) + g(x),
\]
with $|\gamma|_{p} = r/R$,$\quad g(x) \in H(D^{+}(a_{1},R))\quad$and$\quad\|g(x)\| < r$.
\end{theorem}
\begin{proof}
We'll begin with the \textsl{only if} part, prove the theorem in a special case and then generalize.\\
Let $F(x)$ be a one-to-one, onto mapping satisfying the conditions of the theorem with
\[
a = a_{1} = 0, \quad \text{and} \quad r = R = 1.
\]

There is an $x_{1}$ such that $F'(x_{1}) \neq 0$ and $|x_{1}|_{p} < 1$. Since $F(0) = 0$, we can say $|F(x_{1})|_{p} < 1$.\\
Let
\[
	y(x) = F(x + x_{1}) - F(x_{1}).
\]
From the given conditions on $F(x)$, it is simple to show that $y(x)$ is a one-to-one, onto mapping from $D^{+}(0,1)$ to itself, $y(0) = 0$ and $y'(0) \neq 0$.
Let
\[
	y(x) = \gamma_{0}x + \sum_{n \geq 2}c_{n}x^{n}.
\]
If any $|c_{n}|_{p} = 1$, let $N$ be the largest such $n$. The Newton polygon for $y(x)$ shows $0$ must be a root of multiplicity $N$. Since $y'(0) \neq 0$, this cannot be.
Thus
\[
	y(x) = \gamma_{0}x + h(x), \qquad \textup{with} \ \|h\| < 1.
\]
We have
\[
	F(x) = F(x_{1}) + y(x - x_{1}) = F(x_{1}) + \gamma_{0}(x - x_{1}) + h(x - x_{1})
\]
\[
= \gamma_{0}x + F(x_{1}) - \gamma_{0}x_{1} + h(x - x_{1}).
\]
Let
\[
	g_{0}(x) = F(x_{1}) - \gamma_{0}x_{1} + h(x - x_{1}).
\]
We have $g_{0}(x) \in H(D^{+}(0,1))$ and, because of the way $x_{1}$ was chosen, $\|g_{0}\| < 1$.\\

$F(x) = \gamma_{0}x + g_{0}(x)$ and since $\|g_{0}\| < 1$, we must have $|\gamma_{0}|_{p} = 1$.
This establishes the theorem for $F(x)$.\\

To establish the theorem for a general $f(x)$ satisfying the given conditions, choose $\gamma_{1}$ and $\gamma_{2}$ so that $|\gamma_{1}|_{p} = R$ and $|\gamma_{2}|_{p} = r$ and then let
\[
	F(x) = \gamma_{2}^{-1}\left(f(a_{1} + \gamma_{1}x) - a\right)
\]
It's simple to show $F(x)$ satisfies the conditions for which the theorem has been proved.
Hence,
\[
	F(x) = \gamma_{0}x + g_{0}(x), \qquad \textup{with} \ |\gamma_{0}|_{p} = 1 \  \textup{and} \ \|g_{0}\| < 1.
\]
It follows that
\[
	f(x) = a + \gamma_{2}F\left(\frac{x - a_{1}}{\gamma_{1}}\right) = a + \gamma(x - a_{1}) +  g(x),
\]
where
\[
	\gamma = \gamma_{2}\gamma_{0}/\gamma_{1} \quad \textup{and} \quad g(x) = \gamma_{2}g_{0}\left(\frac{x - a_{1}}{\gamma_{1}}\right).
\]
We have
\[
|\gamma|_{p} = r/R, \quad g(x) \in H(D^{+}(a_{1},R)) \quad \textup{and} \quad \|g(x)\| < r,
\]
as desired.\\

Now we are ready for the converse.\\
Suppose 
\[
	 f \colon D^{+}(a_{1},R) \longmapsto D^{+}(a,r), \quad f(x)\in H(D^{+}(a_{1},R)), \quad f(a_{1}) = a.
\]
and
\[
	f(x) = a + \gamma (x - a_{1}) + g(x),
\]
with $|\gamma|_{p} = r/R$,$\quad g(x) \in H(D^{+}(a_{1},R))\quad$and$\quad\|g(x)\| < r$.\\

In order to show that $f(x)$ is a one-to-one, onto mapping from $D^{+}(a_{1},R)$ to $D^{+}(a,r)$, we can use a lemma, \textsl{Newton Approximation}, from \cite{wS1984}.  The lemma says that if there is an $s\in \mathbb{C}_{p}$ such that 
\[
	\sup \left|\frac{f(x_{1}) - f(x_{2})}{x_{1} - x_{2}} - s\right|_{p} <  |s|_{p}
\]
for all $x_{1}, x_{2} \in D^{+}(a_{1},R)$, $x_{1} \neq x_{2}$, then $f(x)$ is a 1-to-1 mapping from $D^{+}(a_{1},R)$ onto $D(a,R|s|_{p})$.If we use $s = \gamma$ and the definition of $f(x)$, the result is quickly obtained.
\end{proof}
Here's a theorem for integration by substitution.
\begin{theorem}[Substitution]\label{T:substitution}
Suppose $f(x)$ is Krasner analytic on $\mathcal{A}(a,b)$ and $r = |a - b|_{p}$. Let $x(t)$ be a one-to-one, onto mapping from a closed disc $D^{+}(a_{1},R)$ to $D^{+}(a,r)$ with $x(a_{1}) = a$. Let $b_{1} = x^{-1}(b)$. Let $\alpha \in \mathbb{Z}_{\geq 0}$.
Then,
\begin{equation}\label{E:subst}
	\int_{\mathcal{A}(a,b),\Phi_{\alpha}}f(x)\,dx = \int_{\mathcal{A}(x^{-1}(a),x^{-1}(b)),\Phi_{\alpha}}f(x(t))x'(t)\,dt.
\end{equation}
\end{theorem}
\begin{proof}
By Theorem \ref{T:substmap} we have
\[
	x(t) = a + \gamma(t - a_{1}) +  g(t),
\]
with $|\gamma|_{p} = r/R$,$\quad g(t) \in H(D^{+}(a_{1},R))\quad$and$\quad\|g\| < r$.\\

From $x(a_{1}) = a$, we have $g(a_{1}) = 0$.\\

It will be useful to have
\begin{lemma}\label{L:substitution}
If $|t - a_{1}|_{p} = |t - b_{1}|_{p} = R$, then
\[
	\omega\left(\frac{x(a_{1}) - x(t)}{x(a_{1}) - x(b_{1})}\right) = \omega\left(\frac{a_{1} - t}{a_{1} - b_{1}}\right).
\]
\end{lemma}
\begin{proof}
\[
	\frac{x(a_{1}) - x(t)}{x(a_{1}) - x(b_{1})} = \frac{x(t) - x(a_{1})}{x(b_{1}) - x(a_{1})}  = \frac{\gamma(t - a_{1}) + g(t)}{\gamma(b_{1} - a_{1}) + g(b)} = \frac{t - a_{1} + g(t)/\gamma}{b_{1} - a_{1} + g(b)/\gamma}.
\]
Since $|g(t)/\gamma|_{p} < R$,
\[
\omega\left(\frac{x(a_{1}) - x(t}{x(a_{1}) - x(b_{1})}\right) = \omega\left(\frac{a_{1} - t}{a_{1} - b_{1}}\right).
\]
\end{proof}
Now we're ready to show that integrals of functions of the form $f(x) = (x - x_{0})^{-1}$ can be evaluated by the given substitution $x = x(t)$. If $x_{0} \notin D^{+}(a,r)$, then
\[
	\frac{1}{x - x_{0}} \in H(D^{+}(a,r))\quad\textup{and}\quad f(x(t))x'(t) = \frac{x'(t)}{x(t) - x_{0}} \in H(D^{+}(a_{1},R)). 
\]
Therefore, by Theorem \ref{T:holo=0},
\[
	\int_{\mathcal{A}(a,b),\Phi}\frac{1}{x - x_{0}}\,dx = 0 = \int_{\mathcal{A}(a_{1},b_{1}),\Phi}\frac{x'(t)}{x(t) - x_{0}}\,dt.
\]
If $x_{0} \in D^{+}(a,r)$, we can write, noting that because $x(t) - x_{0}$ is a 1-to-1 mapping,
\[
	x(t) - x_{0} = (t - t_{0})h(t),
\]
where  $x(t_{0}) = x_{0}$ and $h(t) \neq 0$ for any $t \in D^{+}(a_{1},R)$.
Differention leads to
\[
	\frac{x'(t)}{x(t) - x_{0}} = \frac{1}{t - t_{0}} + \frac{h'(t)}{h(t)}.
\]
The function $h'/h \in H(D^{+}(a_{1},R)$, so
\[
	\int_{\mathcal{A}(a_{1},b_{1}),\Phi}\frac{x'(t)}{x(t) - x_{0}}\,dt = \int_{\mathcal{A}(a_{1},b_{1}),\Phi}\frac{1}{t - t_{0}}\,dt = \frac{1}{1 -\omega^{p^{\alpha}}\left(\frac{a_{1} - t_{0}}{a_{1} - b_{1}}\right)}.
\]
On the other hand,
\[
	\int_{\mathcal{A}(a,b),\Phi}\frac{1}{x - x_{0}}\,dx = \frac{1}{1 -\omega^{p^{\alpha}}\left(\frac{a - x_{0}}{a - b}\right)} = \frac{1}{1 -\omega^{p^{\alpha}}\left(\frac{x(a_{1}) - x(t_{0})}{x(a_{1}) - x(b_{1})}\right)}.
\]
Application of Lemma \ref{L:substitution} establishes Theorem \ref{T:substitution} for the function
\[
	f(x) = \frac{1}{x - x_{0}}.
\]
If $f(x) = (x - x_{0})^{-n}$ and $n \neq -1$, then Theorem \ref{T:intf'=0} shows each side of equation \eqref{E:subst} is $0$. This establishes the theorem for rational functions.
Theorem \ref{T:unif conv} can be applied to establish the theorem for Krasner analytic functions.
\end{proof}
This next result gives a simple condition on the coefficients of a power series that allows an instant calculation of its integral.
\begin{theorem}\label{T:int,lim=L}
	Suppose $f \in H(\mathcal{A}(a,b))$ and $f$ has controlled coefficients of order $\beta < 1$. Let $\Phi$ be an interlocked family of path sequences and
\[
	f(x) = \sum_{n \geq 0}c_{n}(x - b)^{n}.
\]
If 
\[
	\lim_{n \to -1}c_{n}(a - b)^{n + 1} = L \qquad  (n \to -1\ \text{in}\ \mathbb{Z}_{p}.),
\]
then
\[
	\int_{\mathcal{A},\Phi}f(x)\,dx = -L.
\]
\end{theorem}
\begin{proof}
	
	By Theorem \ref{T:varphi=kPhi}, it is sufficient to just consider $\varphi(k) = k$. Let $\theta(k) = [\log_{p}k]$. By Theorem \ref{T:best A_f,phi1}, we have
\begin{align*}
	\int_{\mathcal{A},k}f(x)\,dx & = - \lim_{k \to \infty}\sum_{\substack{n = p^{k} - 1\\ n \equiv -1 \pmod{p^{\theta(k)}}}}^{n_{k} - 1}c_{n}(a - b)^{n+1} \sum_{j \geq 1}(-1)^{j - 1}\binom{n}{jp^{k} - 1}\\
&= -L\lim_{k \to \infty}\sum_{\substack{n = p^{k} - 1\\ n \equiv -1 \pmod{p^{\theta(k)}}}}^{n_{k} - 1} \sum_{j \geq 1}(-1)^{j - 1}\binom{n}{jp^{k} - 1}\\
&= L\int_{\mathcal{A}(1,0),k}\frac{1}{1 - x}\,dx = -L.
\end{align*}

\end{proof}
\noindent Here's an application of Theorem \ref{T:int,lim=L}.\\

Let $\psi_{p}(x)$ be the derivative of the p-adic log gamma function on $\mathbb{C}_{p}\setminus\mathbb{Z}_{p}$.  (This log gamma function is denoted $\log \Gamma_{p}(x)$ in  \cite{hC2007} and $G_{p}(x)$ in \cite{jD1977}.) A series expansion around $x = 1/p\,$ is found in \cite{jD1977}. Differentiation and a slight notational change yield
\[
	\psi_{p}'(x) = \sum_{n \geq 0}(-1)^{n}(n + 1)L_{p}^{*}(n + 2)p^{n + 2}(x - \frac{1}{p})^{n}, \qquad \left| x - \frac{1}{p} \right|_{p} < p.
\]
Here, 
\[
	L_{p}^{*}(s) = \frac{1}{p - 1}\sum_{\chi}L_{p}(s,\chi),
\]
where the sum is over all Dirichlet characters mod $p$.\\

Let $j \in \mathbb{Z}_{\geq 0}$. Then
\[
	\left(x - \frac{1}{p}\right)^{j}\psi_{p}'(x) =  \sum_{n \geq j}(-1)^{n - j}(n - j + 1)L_{p}^{*}(n - j + 2)p^{n - j + 2}(x - \frac{1}{p})^{n}.
\]
Theorem \ref{T:int,lim=L} can be applied with $a = 0$ and $b = 1/p$. The result is
\begin{theorem}\label{T:intwithL*}
Let $\varphi(k) = k$ and $j \in \mathbb{Z}_{\geq 0}$.
\[
	\int_{\mathcal{A}(0,\frac{1}{p}),\varphi}\left(x - \frac{1}{p}\right)^{j}\psi_{p}'(x)\,dx = (-p)^{1 - j}jL_{p}^{*}(1 - j).
\]
\end{theorem}

Now let's go to a general sufficient condition for a line integral of a function $f(x)$, holomorphic on an arc $\mathcal{A}(a,b)$, to exist.
\begin{theorem}[Sufficiency]\label{T:sufficiency}
	Suppose $f(x) = \sum_{n \geq 0}c_{n}(x - b)^{n} \in H(\mathcal{A}(a,b))$ and $f(x)$ is bounded or has controlled coefficients of order $\beta < 1$. 
	Let $\varphi$ be a path sequence satisfying
\[
	\varphi(k) - (\beta + 1)\nabla \varphi(k) - \beta \log_{p} \varphi(k) \to \infty \quad \text{as} \  k \to \infty.
\]
	Let $\theta$ be an auxiliary function satisfying
\[
	\varphi(k) - \beta\varphi(k + 1) - \beta \log_{p} \varphi(k) -\theta(k) \to \infty \quad \text{as} \  k \to \infty.
\]
	Furthermore suppose that, given any $\varepsilon > 0$, there is a $K_{\varepsilon}$ such that for $k > K_{\varepsilon}$ and all $d \in \mathbb{Z}^{+}$ we have 
\[
		\left| c_{dp^{\theta(k) + \nabla\varphi(k)} - 1}\left(a - b\right)^{dp^{\theta(k) + \nabla\varphi(k)}} - c_{dp^{\theta(k)} - 1}\left(a - b\right)^{dp^{\theta(k)}}\right|_{p} < \varepsilon,
\]
then
\[
	\int_{\mathcal{A}(a,b),\varphi}f(x)\,dx \qquad \text{exists.}
\]	
\end{theorem}
\begin{corollary}\label{C:suff} Suppose $\varphi(k) = \alpha + \lambda k$, $ \alpha \in \mathbb{Z}$, $ \lambda\in \mathbb{Z}^{+}$. Also suppose $f(x)$ has controlled coefficiens of order $\beta < 1$. If 
\[
	\lim_{k \to \infty}c_{dp^{\varphi(k)} - 1}\left(a - b\right)^{dp^{\varphi(k)}}
\]
exists for each $d \in \mathbb{Z}^{+}$, and the convergence is uniform over $d$, then
\[
	\int_{\mathcal{A}(a,b),\varphi}f(x)\,dx \qquad \text{exists.}
\]	
\end{corollary}
Note that if $\varphi$ is a polynomial, the condition on $\varphi$ is satisfied.
The condition on $\theta$ says that while $\theta(k)$ goes to infinity, it stays sufficiently less than $\varphi(k)$.
\begin{proof}
$p$ is an odd prime. We will show, Definition \ref{D:1st int}, that $lim_{k \to \infty}A_{f,\varphi}(k)$ exists. The p-adic Cauchy criterion is that 
\[
	\lim_{k \to \infty}\left|A_{f,\varphi}(k + 1) - A_{f,\varphi}(k)\right|_{p} = 0.
\]
Equations \eqref{E:best A_f,phi1} and \eqref{E:best A_f,phi2} in Theorems \ref{T:best A_f,phi1} and \ref{T:best A_f,phi2} will be used.
Writing $n = dp^{\theta(k)} - 1$, equation $\eqref{E:best A_f,phi1}$ can be rewritten as
\[
A_{f,\varphi}(k) = - \sum_{d = p^{\varphi(k) - \theta(k)}}^{n_{k}p^{-\theta(k)}}c_{n}(a - b)^{n+1} \sum_{j \geq 1}(-1)^{j - 1}\binom{dp^{\theta(k)} - 1}{jp^{\varphi(k)} - 1} + \eta_{1}(k).
\]
Now if we let $\theta_{2}(k) = \theta(k) + \nabla \varphi(k)$, and use $n' = dp^{\theta(k) + \nabla \varphi(k)}$, equation $\eqref{E:best A_f,phi2}$ becomes
\[
A_{f,\varphi}(k + 1) = - \sum_{d = p^{\varphi(k) - \theta(k)}}^{n_{k}p^{-\theta(k)}}c_{n'}(a - b)^{n'+1} \sum_{j \geq 1}(-1)^{j - 1}\binom{dp^{\theta(k) + \nabla \varphi(k)} - 1}{jp^{\varphi(k + 1)} - 1} + \eta_{2}(k).
\]
$\eta_{1}(k)$ and $\eta_{2}(k) \to 0$ as $k \to \infty$.
Because the indices of summation are the same, we can write $A_{f,\varphi}(k + 1) - A_{f,\varphi}(k)$ as a single (double)summation and then use the identity $x_{1}y_{1} - x_{2}y_{2} = x_{1}(y_{1} - y_{2}) + y_{2}(x_{1} - x_{2})$ to obtain
\[
	A_{f,\varphi}(k) - A_{f,\varphi}(k + 1) = S_{1} + S_{2} + \eta(k),
\]
with \[
	S_{1} = \sum_{d}\sum_{j}(-1)^{j - 1}\left(c_{dp^{\theta(k) + \nabla \varphi(k)} -1}(a - b)^{dp^{\theta(k) + \nabla \varphi(k)}} - c_{dp^{\theta(k)} - 1}(a - b)^{dp^{\theta(k)}}\right)\binom{dp^{\theta(k)} - 1}{jp^{\varphi(k)} - 1}
\]
and
\[
	S_{2} = \sum_{d}\sum_{j}(-1)^{j - 1}c_{dp^{\theta(k) + \nabla \varphi(k)} -1}(a - b)^{dp^{\theta(k) + \nabla \varphi(k)}}\left(\binom{dp^{\theta(k) + \nabla\varphi(k)} - 1}{jp^{\varphi(k + 1)} - 1} - \binom{dp^{\theta(k)} - 1}{jp^{\varphi(k)} - 1}\right).
\]
Now we will show the terms in $S_{2}$ are small, so $S_{2} \to 0$ as $k \to \infty$.
Since $n~=~dp^{\theta(k) + \nabla \varphi(k)} \leq m_{k}$, it follows that
\[
	\left|c_{n}(a - b)^{n}\right|_{p} <Mn^{\beta} \leq Mm_{k}^{\beta} = const\cdot\varphi^{\beta}(k)p^{\beta\varphi(k + 1)}.
\]
	In order to bound the differences of the binomial coefficients, a simpler, modified, version of a Kazandzidis theorem, see \cite{aR2000}, will be used. This lemma is somewhat stronger than a similar lemma of F.~Baldassari.
\begin{lemma}\label{L:Kazmodified}
Let $a$, $b$, $t \in \mathbb{Z}^{+}$. $p$ is an odd prime. Then,
\[
	\left|\binom{ap^{t} - 1}{bp^{t} - 1} - \binom{a - 1}{b - 1}\right|_{p} \leq p^{-3}|b|_{p}^{2}|a - b|_{p}. \qquad (\text{Use}\ p^{2} \text{instead of}\ p^{3} \text{if}\ p = 3.)
\]
\end{lemma}
\begin{proof}
Kazandzidis' theorem is
\[
\binom{ap}{bp}\equiv \binom{a}{b} \pmod{p^{3}ab(a-b)\binom{a}{b}\mathbb{Z}_{p}} \qquad \text{as a statement in} \ \mathbb{Z}_{p}
\]
From this it immediately follows that 
\[
\binom{ap - 1}{bp - 1} \equiv \binom{a -1}{b -1} \pmod{p^{3}b^{2}(a-b)\mathbb{Z}_{p}} 
\]
and the statement of the lemma.
\end{proof}
Applying the lemma to the terms in $S_{2}$, using $t = \nabla\varphi(k)$, yields
\[
	\left|\binom{dp^{\theta(k) + \nabla\varphi(k)} - 1}{jp^{\varphi(k + 1)} - 1} - \binom{dp^{\theta(k)} - 1}{jp^{\varphi(k)} - 1}\right|_{p} \leq p^{-\theta(k) - 2\varphi(k) - 3}.
\]
When the two inequalities are combined we get
\[
	|S_{2}|_{p} \leq const\cdot p^{\beta\log_{p} \varphi(k) + \beta\varphi(k + 1) - 2\varphi(k) - \theta(k)}.
\]
Because of the given hypotheses,the exponent in this last expression approaches $-\infty$ as $k \to \infty$. Hence, $S_{2} \to 0$ as $k \to \infty$.
	In order to have $S_{1} \to 0$ it is sufficient to have the individual terms in the sum be small. Since the binomial coefficients are p-adically bounded by $1$, it is enough to have the differences in expressions $c_{n}(a -b)^{n + 1}$ be small. That is the theorem.
	
	To establish the Corollary, let $\varphi(k) = \alpha + \lambda k$ and $ \theta(k)~=~\varphi([\log_{p} k])$. $\nabla\varphi(k) = \lambda$ and the Corollary follows easily.
\end{proof}

Corollary \ref{C:suff} can be applied to the logarithmic derivative of the Artin-Hasse  exponential function
\[
	(E'/E)(x) = \sum_{n \geq 0}x^{p^{n} - 1}.
\]
In this case we can use $a = 1$ with $b = 0$ and $\varphi(k) = k$.

Applying Theorem \ref{T:sufficiency} means looking at $lim_{k \to \infty}c_{dp^{k} - 1}$. The sequences are constant and the limits are either $0$ or $1$, so we can conclude that
\[
	\int_{\mathcal{A}(1,0),k}(E'/E)(x)\,dx \qquad \text{exists}.
\]
\begin{theorem}
\[
	\int_{\mathcal{A}(1,0),k}(E'/E)(x)\,dx \equiv -1 \pmod{p^{3}}.
\]
\end{theorem}
\begin{proof}
	The only non-zero coefficients in the power series, which are all $1$, occur when $n = p^{t} - 1$ for some $t \in \mathbb{Z}^{+}$.
	By Theorem $\ref{T:1st A_k c_n}$ we have, for $p >3$,
\begin{equation}\label{E:E'/E}
	\int_{\mathcal{A}(1,0),k}(E'/E)(x)\,dx = -\lim_{k \to \infty}\sum_{\substack{n \geq p^{k} - 1\\ n = p^{t} - 1}} \sum_{j \geq 1}(-1)^{j - 1}\binom{p^{t} - 1}{jp^{k} - 1}.
\end{equation}
The term with $t = k$ has the value $-1$. If $t > k$, Lemma \ref{L:Kazmodified} yields
\[
	\binom{p^{t} - 1}{jp^{k} - 1} \equiv \binom{p^{t - k} - 1}{j - 1} \pmod{p^{3}}.
\]
The rightmost sum in Equation \eqref{E:E'/E} is then congruent, $\pmod{p^{3}}$, to
\[
	\sum_{j = 1}^{p^{t - k}}(-1)^{j - 1}\binom{p^{t - k} - 1}{j - 1} = \sum_{m = 0}^{p^{t - k} - 1}(-1)^{m}\binom{p^{t - k} - 1}{m} = 0.
	\]
If $p = 3$, this proof doesn't quite give the result, but numerical calculation shows the theorem is valid for $p = 3$.
\end{proof}
Numerical calulations for $p = 3$ and $5$ show no repeating pattern for the first few non-zero p-adic digits.\\

The following theorem shows that often if the integral of a function exists with $\varphi(k)$, then the same function can be integrated with $\alpha + \varphi(k)$.
\begin{theorem}\label{T:add alpha suff}
	If $f$, $\varphi$ and $\theta$ satisfy the conditions of Theorem \ref{T:sufficiency} and $\varphi_{1}(k)~=~\alpha~+~\varphi(k)$, then $f$, $\varphi_{1}$ and $\theta$ will satisfy the conditions of Theorem \ref{T:sufficiency}.
\end{theorem}
\begin{proof}
	Noting that $\nabla\varphi_{1}(k) = \nabla\varphi(k)$, it is simple to show that the expressions needed for the inequalities differ by a bounded amount if $\varphi(k) > \alpha$.
\end{proof}
	Here's a necessary and sufficient condition for a function of a very particular form to be integrable using a non-linear $\varphi$.
\begin{theorem}
	Suppose $f \in H(\mathcal{A}(a,b))$, $f$ has controlled coefficients of order $\beta$ and 
\[
	f(x) = \sum_{n > k_{0}}c_{p^{\varphi(n) -1}}(x - b)^{p^{\varphi(n) -1}},
\]	
where $\varphi(n) > 0$ and is strictly increasing for $n > k_{0}$ and $\nabla \varphi(k) > \log_{p} (\beta + 2)\varphi(k)$ for $k > k_{0}$.
Then,
\[
	\int_{\mathcal{A}(a,b),\varphi}f(x)\,dx 
\]
exists if and only if $f$ is bounded \textup{(} $\beta = 0$ \textup{)} and 
\[
	\lim_{k \to \infty}c_{p^{\varphi(k) -1}}(a - b)^{p^{\varphi(k)}} = L.
\]
Then,
\[
	\int_{\mathcal{A}(a,b),\varphi}f(x)\,dx = -L.
\]
\end{theorem}
\begin{proof}
	Theorem \ref{T:n_k1} gives us
	\[
		A_{f,\varphi}(k) = - \sum_{n = p^{\varphi(k)} - 1}^{n_{k} - 1}c_{n}(a - b)^{n+1} \sum_{j 	\geq 1}(-1)^{j - 1}\binom{n}{jp^{\varphi(k)} - 1} + \eta(k).
\]
The sum begins with $n = p^{\varphi(k)} - 1$. The next possible non-zero $c_{n}$ is $n = p^{\varphi(k + 1)} - 1$. However,
\[
	p^{\varphi(k + 1)} = p^{\varphi(k)}p^{\nabla \varphi(k)} > p^{\varphi(k)}(\beta + 2)\varphi(k) >  p^{\varphi(k)}(\beta + 2)\left(\frac{p - 1}{p}\right)\varphi(k) = n_{k}.
\]
Hence, the only possible non-zero term in the sum is at $n = p^{\varphi(k)} - 1$.
Then, 
\[
	A_{f,\varphi}(k) = - c_{p^{\varphi(k)} - 1}(a - b)^{p^{\varphi(k)}} \sum_{j \geq 1}(-1)^{j - 1}\binom{p^{\varphi(k)} - 1}{jp^{\varphi(k)} - 1} + \eta(k).
\]
The sum over $j$ is just $1$, so the theorem follows. The fact that $f$ is bounded follows from the existence of $L$ and
\[
	\sup_{x}|f(x)|_{p} = \sup_{n}|c_{n}|_{p}R^{n} < \infty.
\]
\end{proof}

Now let's try to integrate the same type of function $f$, but use the path sequence $\varphi_{1}(k) = k$ in the integral. Let's assume that $f$ can be integrated with $\varphi$ and the integral is $-L \neq 0$. The range of the index of summation is now smaller, so the formula for $A_{f,\varphi_{1}}(k)$ will sometimes give a value near $L$ and sometimes be empty. Thus there is no limit for $A_{f,\varphi_{1}}(k)$ and no integral.

 \section{Uniform Ray Convergence and Ray Limits}
Theorems \ref{T:best A_f,phi1}, \ref{T:best A_f,phi2} and Corollary \ref{C:suff} suggest examination of an idea similar to approaching a point in the plane along a fixed line.
\begin{definition}\label{D:ray}
Let  $\varphi(k)$ be a path sequence.  A \emph{$\varphi$-ray}, with endpoint $m \in \mathbb{Z}$ and direction $d \in \mathbb{Z}^{+}$ will be a set of the form 
\[
	\lbrace m + dp^{\varphi(k)} \mid k \geq k_{0},\  \text{for some}\  k_{0}\rbrace.
\]
\end{definition}
$m = -1$ will be the primary value of $m$ used with integration.
\begin{definition}\label{D:urc}
Let $\varphi(k)$ be a path sequence. A sequence $a(n)$ in $\mathbb{C}_{p}$, will be called \emph{uniformly $\varphi$-ray convergent at $m$}, written \emph{$\varphi$-urc}, if for each $d \in \mathbb{Z}^{+}$,
\[
	L_{\varphi}(d) = \lim_{k \to \infty}a\left(m + dp^{\varphi(k)}\right) 
\]
exists and the limits are uniform over $d$. These limits will be called \emph{$\varphi$-ray limits}.
\end{definition}
The definitions of $\varphi$-urc and ray limits can be extended to interlocked families. 
\begin{definition}\label{D:Phiurc}
Let $\Phi$ be an interlocked family of path sequences. $a(n)$ is \emph{$\Phi$-urc} if there is a function $L_{\Phi} \colon \mathbb{Z}^{+} \mapsto \mathbb{C}_{p}$ such that, given any $\varepsilon > 0$, there is a $\varphi_{\varepsilon} \in \Phi$ and a $k_{\varepsilon} \in \mathbb{Z}^{+}$ with
\[
	\mid a \big(m + dp^{\varphi_{\varepsilon}(k)}\big) - L_{\Phi}(d)\mid_{p}\  < \  \varepsilon
\]
for all $k > k_{\varepsilon}$ and all $d \in \mathbb{Z}^{+}$.
\end{definition}
The $L_{\Phi}(d)$ will be called \emph{$\Phi$-ray limits}
The following theorem shows that the $L_{\Phi}(d)$ are uniquely defined.
\begin{theorem}
Suppose $a(n)$ is $\Phi$-urc and each of the functions $L_{1}(d)$ and  $L_{2}(d)$ satisfy the conditions of ray limits in the definition of $\Phi$-urc. Then, 
\[
	L_{1}(d) = L_{2}(d)\qquad \textup{for all}\ d \in \mathbb{Z}^{+}.
\]
\end{theorem}
\begin{proof}
Let $\varepsilon > 0$. There are $\varphi_{i,\varepsilon}$ and $k_{i,\varepsilon}$, $i = 1$, $2$, so that
\[
	\mid a \big(m + dp^{\varphi_{i,\varepsilon}(k)}\big) - L_{i}(d)\mid_{p}\  < \  \varepsilon \qquad \textup{if}\ k > k_{i,\varepsilon}.
\]
Let $\varphi_{0,\varepsilon}(k) \in \Phi$ be a common subsequence of $\varphi_{1,\varepsilon}(k)$ and $\varphi_{2,\varepsilon}(k)$ for large $k$. Then there is a $k_{0}$ such that
\[
	\mid a \big(m + dp^{\varphi_{0,\varepsilon}(k_{0})}\big) - L_{i}(d)\mid_{p}\  < \  \varepsilon \qquad \textup{for} \ i = 1, 2.
\]
Now we have
\[
	|L_{1}(d) - L_{2}(d)|_{p} = |L_{1}(d) - a \big(m + dp^{\varphi_{0,\varepsilon}(k_{0})}\big) + a \big(m + dp^{\varphi_{0,\varepsilon}(k_{0})}\big) - L_{2}(d)|_{p} < \varepsilon.
\]
Since $\varepsilon$ could be any positive number, $L_{1}(d) = L_{2}(d) \quad \textup{for all}\ d \in \mathbb{Z}^{+}$.
\end{proof}

The next five results follow easily from the definitions.
\begin{theorem}\label{T:varphi in Phi}
If $\Phi$ is an interlocked family of path sequences, $\varphi \in \Phi$ and $a(n)$ is $\varphi$-urc, then $a(n)$ is $\Phi$-urc with
\[
	L_{\Phi}(d) = L_{\varphi}(d), \quad d \in \mathbb{Z}^{+}.
\]
\end{theorem}
\begin{theorem}\label{T:Phiseqadd}
If $a(n)$ and $b(n)$ are $\Phi$-urc at $m$, then $a(n) + b(n)$ and $ca(n)$, $c \in \mathbb{C}_{p}$, are $\Phi$-urc at $m$. If  $a(n)$ and $b(n)$ are $\Phi$-urc at $m$ and bounded, then $a(n)b(n)$ is  $\Phi$-urc at $m$. The corresponding statements about the ray limits also hold.
\end{theorem}
\begin{theorem}
If $\Psi$ is an interlocked family of path sequences and $a(n)$ is $\Psi$-urc at $m$, and for some $\alpha \in \mathbb{Z}^{+}$ we write
\[
	\Phi(k) = \{\,\varphi \mid \varphi(k) = \alpha + \varphi_{1}(k),\  \varphi_{1} \in \Psi, \}
\]	
then $a(n)$ is $\Phi$-urc at $m$ and $L_{\Phi}(d) = L_{\Psi}(dp^{\alpha})$ for all $d$.
\end{theorem}
\begin{theorem}
If $a(n)$ is $\varphi_{1}$-urc at $m$, $\varphi_{1}(k) = \alpha_{1} + \lambda k$ and we write $\varphi(k) = \alpha + \varphi(k)$, $\alpha \in \mathbb{Z}$, then $a(n)$ is $\varphi$-urc at $m$.
\end{theorem}
\begin{theorem}
If $\varphi_{1}(k) = \alpha_{1} + \lambda k$ and $\varphi_{2}(k) = \alpha_{2} + \lambda k$ and $ \alpha_{1} \equiv  \alpha_{2} \pmod{\lambda}$ and $a(n)$ is $\varphi_{1}$-urc at $m$ with ray limits $L_{\varphi_{1}}(d)$, then $a(n)$ is $\varphi_{2}$-urc at $m$ and $L_{\varphi_{1}}(d) = L_{\varphi_{2}}(d)$.
\end{theorem}
\begin{proof} 
$\varphi_{1}$ and $\varphi_{2}$ have the same set of values for large $k$.
\end{proof}
\begin{theorem}\label{T:Phiunifconv}
If the $a_{i}(n)$ are $\Phi$-urc at $m$ and $a_{i}(n) \to a(n)$ uniformly over $n$ as $i \to \infty$, then $a(n)$ is $\Phi$-urc at $m$. Furthermore, if we write $L_{i}(d)$ for the ray limits of the sequence $a_{i}(n)$ and $L_{\Phi}(d)$ for the ray limits of the sequence $a(n)$,
then
\[
	L_{\Phi}(d) = \lim_{i \to \infty}L_{i}(d).
\]
\end{theorem}
\begin{proof}
First we'll show $\lim_{i \to \infty}L_{i}(d)$ exists for each $d \in \mathbb{Z}^{+}$. The proof will also show that the convergence is uniform with respect to $d$. Let $\varepsilon > 0$. Choose $I$ so that $|a_{i + 1}(n) - a_{i}(n)|_{p} < \varepsilon$ for all $n$ and any $i > I$.\\
For any  $i > I$ there are $\varphi_{i,\varepsilon}$ and $\varphi_{i + 1,\varepsilon}$ such that
\[
	\mid a_{i} \big(m + dp^{\varphi_{i,\varepsilon}(k)}\big) - L_{i}(d)\mid_{p}\  < \  \varepsilon \quad \textup{and} \quad \mid a_{i + 1} \big(m + dp^{\varphi_{i + 1,\varepsilon}(k)}\big) - L_{i + 1}(d)\mid_{p}\  < \  \varepsilon
\]
for large $k$.\\
There will also be a $\psi_{i,\varepsilon}(k) \in \Phi$ so that $\psi_{i,\varepsilon}(k)$ is a subsequence of both $\varphi_{i,\varepsilon}$ and $\varphi_{i + 1,\varepsilon}$ for large $k$.
\begin{multline}\label{E:Phiunifconv}
	|L_{i + 1}(d) - L_{i}(d)|_{p}\\= |L_{i + 1}(d) - a_{i + 1} \big(m + dp^{\varphi_{i + 1,\varepsilon}(k)}\big) + a_{i + 1} \big(m + dp^{\varphi_{i + 1,\varepsilon}(k)}\big)\\- a_{i} \big(m + dp^{\varphi_{i,\varepsilon}(k)}\big) + a_{i} \big(m + dp^{\varphi_{i,\varepsilon}(k)}\big) - L_{i}(d)|_{p}.
\end{multline}
Now a value of $k$ can be picked large enough so that the first and last pairs of terms on the right side of \eqref{E:Phiunifconv} have norm $< \varepsilon$ and the center pair can be replaced by
\[a_{i + 1} \big(m + dp^{\psi{i,\varepsilon}(k')}\big)- a_{i} \big(m + dp^{\psi{i,\varepsilon}(k')}\big).
\]
This last expression also has norm $< \varepsilon$, so for $i > I$,
\[
	|L_{i + 1}(d) - L_{i}(d)|_{p} < \varepsilon
\]
and
\[
	\lim_{i \to \infty}L_{i}(d) = L_{\Phi}(d)
\]
exists for each $d \in \mathbb{Z}^{+}$.\\

Now we're ready to show that $a(n)$ is $\Phi$-urc.\\
Given $\varepsilon > 0$, we can choose an $I$ so that
\[
	|L_{I}(d) - L_{\Phi}(d)|_{p} < \varepsilon \qquad \textup{for all}\ d.
\]
and
\[
|a(n) - a_{I}(n)|_{p} <\varepsilon
\]
for all $n$.

There is a $\varphi_{I,\varepsilon} \in \Phi$ and a $k_{I,\varepsilon}$ such that
\[
	\mid a_{I} \big(m + dp^{\varphi_{I,\varepsilon}(k)}\big) - L_{I}(d)\mid_{p}\  < \  \varepsilon
\]
for $k >k_{I,\varepsilon}$ and all $d$.\\
Putting the last three inequalities together with $n = m + dp^{\varphi_{I,\varepsilon}(k)}$ yields
\[
	\mid a \big(m + dp^{\varphi_{I,\varepsilon}(k)}\big) - L_{\Phi}(d)\mid_{p}\  < \  \varepsilon
\]
for all $k > k_{I,\varepsilon}$ and all $d \in \mathbb{Z}^{+}$.
\end{proof}
It will be useful here to define a symbol for an idea that occured in Theorem~\ref{T:int,lim=L}.\\
Let $n \in \mathbb{Z}^{+}$, $m \in \mathbb{Z}_{p}$, $L \in \mathbb{C}_{p}$ and $a(n)$ be a sequence in $\mathbb{C}_{p}$. We define
\[
	\lim_{n \uparrow m} a(n) \overset{\textup{def}}{=} L
\]
if, given an $\varepsilon > 0$, there is a $\delta > 0$ such that if $0 <|n - m|_{p} < \delta$, then
\[
	|a(n) - L|_{p} < \varepsilon.
\]
In other words, if $n$ approaches $m$ in $\mathbb{Z}_{p}$, then $a(n)$ approachs $L$.\\

\noindent Clearly if $\lim_{n \to \infty}a(n) = L$, then
\[
	\lim_{n \uparrow m} a(n) = L \qquad\textup{for any}\ m \in \mathbb{Z}_{p}.
\]
\begin{theorem}
If
\[
	\lim_{n \uparrow m} a(n) = L,
\]
then $a(n)$ is $\Phi$-urc at $m$ for any interlocked set $\Phi$ and, for all $d$,
\[
	L_{\Phi}(d) = L.
\]
\end{theorem}
\begin{proof}
Let $\varphi \in \Phi$. Given $\varepsilon > 0$ there is a $\delta > 0$ such that
\[
	|a(n) - L|_{p} < \varepsilon \qquad \textup{if}\ 0 <|n - m|_{p} < \delta.
\]
Choose $k_{\varepsilon}$ so that
\[
	p^{-\varphi(k_{\varepsilon})} < \delta.
\]
Then if $k > k_{\varepsilon}$ we have
\[
	|a(m + dp^{\varphi(k)}) - L|_{p} < \varepsilon.
\]
Reference to Theorem \ref{T:varphi in Phi} completes the proof.
\end{proof}
\begin{corollary}
If $a(n) = \tau (n)$, for  $n\in \mathbb{Z}^{+}$, and $\tau \colon \mathbb{Z}_{p} \to \mathbb{C}_{p}$ is continuous, then $a(n)$ is $\Phi$-urc at $m$ for all $\Phi$ and for all $m \in \mathbb{Z}$ and
\[
	L_{\Phi}(d) = \tau (m)\qquad \text{for all} \ d \in \mathbb{Z}^{+}.
\]
\end{corollary}
\begin{theorem}\label{T:raylim=L}
If for some $\Phi$, $a(n)$ is $\Phi$-urc at $m$ and $L_{\Phi}(d) = L$ for all $d \in \mathbb{Z}^{+}$, then
\[
		\lim_{n \uparrow m} a(n) = L.
\]
\end{theorem}
\begin{proof}
Given $\varepsilon > 0$ there is a $\varphi_{\varepsilon}$ and a $k_{\varepsilon} > m$ such that if $k > k_{\varepsilon}$ and $d \in \mathbb{Z}^{+}$
\[
	|a(m + dp^{\varphi_{\varepsilon}(k)}) - L|_{p} < \varepsilon.
\]
Let $\delta = p^{-\varphi_{\varepsilon}(k_{\varepsilon})}$. If $|n - m|_{p} < \delta$, then
\[
	n = m + dp^{\varphi_{\varepsilon}(k)}
\]
with $k > k_{\varepsilon}$ and $d \in \mathbb{Z}^{+}$. The proof follows.
\end{proof}

\begin{corollary}
If $a(n)$ is $\Phi_{1}$-urc at $m$ for some $\Phi_{1}$ and some $m$, and $L_{\Phi_{1}}(d) = L$ for all $d \in \mathbb{Z}^{+}$, then $a(n)$ is $\Phi$-urc at $m$ for all $\Phi$ and $L_{\Phi}(d) = L$ for all $d$.
\end{corollary}
\begin{proof}
Combine the two previous theorems.
\end{proof}

\begin{theorem}\label{T:seq mult}
Suppose $a(n) \to 0$ as $n \to \infty$, $|b(n)|_{p} $ is bounded and $b(n)$ is $\varphi$-urc for all $m \leq m_{0},\ m_{0} \in \mathbb{Z}$. Then
\[
	c(n) = \sum_{j = 0}^{n}a_{j}b_{n - j}
\]
is  $\varphi$-urc for all $m \leq m_{0}$.
\end{theorem}
\begin{proof}
For a given $m$, $m \leq m_{0}$, and $\varepsilon> 0$, we have
\begin{multline*}
	c(m + dp^{\varphi(k + 1)}) - c(m + dp^{\varphi(k)})\\ 
= \sum_{j = 0}^{m + dp^{\varphi(k + 1)}}a_{j}b_{m + dp^{\varphi(k + 1)} - j} - \sum_{j = 0}^{m + dp^{\varphi(k)}}a_{j}b_{m + dp^{\varphi(k)} - j}\\
= \eta + \sum_{j = 0}^{j_{\varepsilon}}a_{j}\left( b_{m + dp^{\varphi(k + 1)} - j} - b_{m + dp^{\varphi(k)} - j}\right),
\end{multline*}
where $j_{\varepsilon}$ was chosen so $|a_{j}|_{p} < \varepsilon/M$ for $j > j_{\varepsilon}$. Hence, $|\eta|_{p} < \varepsilon$. $k$ will be chosen large enough so $m + dp^{\varphi(k)} > j_{\varepsilon}$. 
If $k$ is large enough, the terms in the sum will also be smaller than $\varepsilon$.
Since $\varepsilon$ was arbitrary, the ray limits exist and convergence is uniform over $d$.
\end{proof}
	
Now let's look at functions $f \in H(\mathcal{A}(a$, $b)$ and transfer some results on $a(n)$ to results about $f$. There are complications when $m \neq -1$, but for applications to line integrals we can stick to $m = -1$.\\
\begin{definition}\label{D:furc} We will say that
\[
		f(x) = \sum_{n \geq 0}c_{n}(x - b)^{n}
\]
is \emph{$\varphi$-urc} if the sequence
\[
	a(n) = c_{n}(a - b)^{n + 1}
\]
is $\varphi$-urc at $-1$ and \emph{$\Phi$-urc} if $a(n)$ is $\Phi$-urc at $-1$.
\end{definition}
Before we look at the dependence of this definition on the choice of $a$ and $b$, let's restate Corollary \ref{C:suff} in this new language.\\
\begin{theorem}[= Corollary \ref{C:suff}]
Suppose $\varphi(k) = \alpha + \lambda k$, $ \alpha \in \mathbb{Z}$, $ \lambda\in \mathbb{Z}^{+}$. If
\[
		f(x) = \sum_{n \geq 0}c_{n}(x - b)^{n},
\]
$f$ has controlled coefficients of order $\beta < 1$ and $f$ is $\varphi$ -urc, then
\[
	\int_{\mathcal{A}(a,b),\varphi}f(x)\,dx \qquad \text{exists}.
\]
\end{theorem}
\begin{theorem}
Suppose $f \in H(\mathcal{A}(a$, $b)$, $f$ is bounded and 
\[
	f(x) = \sum_{n \geq 0}c_{n}(x - b)^{n}.
\]
If $f$ is $\varphi$-urc, $|a - a'|_{p} < R = |a - b|_{p}$, and  $|b - b'|_{p} < R$ and
\[
	f(x) = \sum_{n \geq 0}c'_{n}(x - b')^{n}, \qquad \text{then}
\]
\[
	\lim_{k \to \infty}c'_{dp^{\varphi(k)} - 1}(a' - b')^{dp^{\varphi(k)}} = \lim_{k \to \infty}c_{dp^{\varphi(k)} - 1}(a - b)^{dp^{\varphi(k)}}.
\]
\end{theorem}
Thus for bounded, $\varphi$-urc holomorphic functions on $\mathcal{A}(a,b)$ we can define
\[
	\mathcal{L}_{f,\varphi}(d) = \lim_{k \to \infty}c_{dp^{\varphi(k)} - 1}(a - b)^{dp^{\varphi(k)}}.
\]
without regard to the choice of center for the circle or choice of $b$ on the arc.
\begin{proof}
This is essentially a simpler version of the proof of Theorem \ref{T:change a,b} with $\beta = 0$. First only $b$ will be changed.
Equation \eqref{E:changeb}, where $b' - b = t(a - b)$, gives us
\[
	c_{n}'(a - b)^{n + 1} = c_{n}(a - b)^{n + 1} + \sum_{m \geq 1}c_{m + n}\binom{m + n}{n}(a - b)^{m + n  + 1}t^{m}.
\]
Let $n = dp^{\varphi(k)} - 1$.
On the right side, the large sum will go to $0$ as $k \to \infty$ because, as in the proof of Theorem \ref{T:change a,b}, we have
\[
	\left|\binom{m + n }{n}\right|_{p} \leq \varphi(k)p^{-\varphi(k)} \qquad \text{if}\ m \leq \varphi(k)
\]
and
\[
	|t^{m}|_{p} < (\max(|t|_{p},|p|_{p}))^{-\varphi(k)} \qquad \text{if}\ m > \varphi(k).
\]
On the left side,
\[
	c_{n}'(a - b)^{n  + 1} = c_{n}'(a - b')^{n  + 1}\left(\frac{a - b}{a - b'}\right)^{n  + 1}.
\]
Since the large fraction goes to $1$ as $k \to \infty$, we have
\[
	\lim_{k \to \infty}c'_{dp^{\varphi(k)} - 1}(a - b')^{dp^{\varphi(k)}} = \lim_{k \to \infty}c_{dp^{\varphi(k)} - 1}(a - b)^{dp^{\varphi(k) }}.
\]
The proof for changing $a$ is even simpler. Since the $c_{n}$ don't change, the result follows from the identity
\[
	(a - b)^{n  + 1} = (a' - b)^{n  + 1}\left(\frac{a - b}{a' - b}\right)^{n  + 1} \qquad \text{with}\ n = dp^{\varphi(k)} - 1.
\]
\end{proof}
From Theorems \ref{T:Phiseqadd} and \ref{T:Phiunifconv}, with $m = -1$, we have
\begin{theorem}
$H_{\Phi}(\mathcal{A}) = \{\,f\mid f \in H(\mathcal{A})$, $f$ is $\Phi$-urc\,$\}$ is a vector space and is closed under uniform convergence.
\end{theorem}
\begin{theorem}
If $f$ is holomorphic on the closed disc $|x - a|_{p} < |a - b|_{p}$, then $f$ is $\Phi$-urc\ for all $\Phi$ and all $\Phi$-ray limits are $0$.
\end{theorem}
\begin{proof}
This follows from $lim_{n \to \infty}|c_{n}|_{p}R^{n} = 0$.
\end{proof}
From Theorem \ref{T:seq mult} we have
\begin{theorem}\label{T:f(x)g(x)}
If $f(x) = \sum_{n \geq 0}a_{n}(x - b)^{n}$ converges for $|x - b|_{p} \leq R$ and $g(x) = \sum_{n \geq 0}b_{n}(x - b)^{n}$ is bounded and $\varphi$-urc on $\mathcal{A}$ and the sequence
\[
	n\longmapsto b_{n}(a - b)^{n + 1}
\]
is $\varphi$-urc at every negative integer, then
$f(x)g(x)$ is bounded and $\varphi$ -urc on $ \mathcal{A}$.
\end{theorem}
\begin{proof}
The coefficients of $f(x)g(x)$, expanded around $b$, are
\[
	c_{n} = \sum_{j = 0}^{n}a_{j}b_{n - j}.
\]
To see that $f(x)g(x)$ is $\varphi$-urc we need to examine
\[
	c_{n}(a - b)^{n + 1}.
\]
If we let $A(n) = a_{n}(a - b)^{n}$ and $B(n) = b_{n}(a - b)^{n + 1}$, then Theorem \ref{T:seq mult} applies to give us Theorem \ref{T:f(x)g(x)}.
\end{proof}
 \begin{theorem}
 If $f(x) = \sum_{n \geq 0}c_{n}(x - b)^{n}$ is bounded on $\mathcal{A}$, then $f'(x)$ is $\Phi$-urc for all $\Phi$ and all $\Phi$-ray limits of $f'(x)$ are $0$.
\end{theorem}
\begin{proof}
\[
	f'(x) = \sum_{n \geq 0}(n + 1)c_{n + 1}(x - b)^{n}.
\]
Because $f(x)$ is bounded on $\mathcal{A}$, we have
\[
	 \left|c_{n}(a - b)^{n}\right|_{p} < M,\qquad \text{for all} \ n.
\]
The theorem now follows easily from the definition of uniform $\Phi$-ray convergence. 
\end{proof}

\begin{theorem}
Suppose $f(x) \in H(\mathcal{A})$. $f(x)$ has controlled coefficients with $\beta < 1$. $f$ is $\Phi$-urc and all $\Phi$-ray limits are $L$. Then,
\[
	\int_{\mathcal{A}, \Phi}f(x)\,dx = -L.
\]
\end{theorem}
\begin{proof}
Theorem \ref{T:int,lim=L} and Theorem \ref{T:raylim=L} with $m = -1$ do the job.
\end{proof}
The last theorem immediately shows that ray limits determine the value of an integral in certain circumstances .
\begin{theorem}\label{T:ints same raylims}
If $f$, $g \in H(\mathcal{A})$, and have controlled coefficients with $\beta < 1$ and $f$ and $g$ are $\Phi$-urc  and
\[
	\mathcal{L}_{f,\Phi}(d) = \mathcal{L}_{g,\Phi}(d) \qquad \text{for all}\ d \in \mathbb{Z^{+}},	
\]
then
\[
	\int_{\mathcal{A}, \Phi}f(x)\,dx = \int_{\mathcal{A}, \Phi}g(x)\,dx.	
\]
\end{theorem}
\begin{proof}
Look at $f(x) - g(x)$.
\end{proof}

\noindent The last result raises two questions:\\
\indent Given the set of ray limits of a function, is there another function with the same ray limits that's easier to integrate?\\
\indent Is there a way to express the value of an integral in terms of the ray limits of the function?\\

The following two theorems give answers in a special case. This will allow us to calculate some intergrals whose values involve the Kubota - Leopoldt L-functions. These theorems will be stated in terms of ray limits at any $m$, but only $m = -1$ will be used when they are applied to integrals.

\begin{theorem}\label{T:raylims same}
	Suppose $\varphi(k) = \alpha + \lambda k$, $\alpha$, $\lambda \in \mathbb{Z}$, $\alpha \geq 0$ and $\lambda \geq 1$. Suppose a holomorphic function on the arc $\mathcal{A}(1,0)$ has $\varphi$-ray limits at $m$ we denote by $\mathcal{L}(d,m)$. Suppose $n \in \mathbb{Z}^{+}$ satisfies $ p^{\lambda} \equiv 1 \pmod{n}$ and $\alpha'$ satisfies $\alpha' p^{\alpha} \equiv 1 \pmod{n}$. Suppose
\[
\mathcal{L}(d,m) = \mathcal{L}(d',m) \qquad \text{if}\ d \equiv d' \pmod{n}
\]
and that the definition of $\mathcal{L}(d,m)$ is extended to all $d \in \mathbb{Z}$ by
\[
	\mathcal{L}(d,m) = \mathcal{L}(d',m), \qquad \text{where} \ d' > 0 \ \text{and} \ d \equiv d' \pmod{n}.
\]
Then, the function
\[
	f(x) = \frac{P(x)}{1 - x^{n}},
\]
with $P(x) = \sum_{i = 0}^{n - 1}\mathcal{L}(\alpha'(i - m), m)x^{i}$, has ray limits
\[
	\mathcal{L}_{f, \varphi}(d,m) = \mathcal{L}(d,m).
\]
\end{theorem}
\begin{proof}
\[
	f(x) = P(x)\sum_{r \geq 0}x^{nr} = \sum_{i \geq 0}c_{i}x^{i},
\]
with $c_{i} = \mathcal{L}(\alpha'(i - m), m)$.
Hence,
\[
	\mathcal{L}_{f, \varphi}(d,m) = \lim_{k \to \infty}c_{dp^{\alpha + \lambda k} + m} = \lim_{k \to \infty}\mathcal{L}(\alpha' dp^{\alpha + \lambda k},m) = \mathcal{L}(d,m).
\]
Since these limits are of constant sequences, the convergence is uniform over $d$.
\end{proof}
We can note that $f$ is bounded on $\mathcal{A}(1,0)$.
\begin{theorem}\label{T:intP/1-x^n}
	Suppose $\varphi(k) = \alpha + \lambda k$, $\alpha$, $\lambda \in \mathbb{Z}$, $\alpha \geq 0$ and $\lambda \geq 1$. Suppose $n \in \mathbb{Z}^{+}$ satisfies $ p^{\lambda} \equiv 1 \pmod{n}$. Let
\[
	f(x) = \frac{P(x)}{1 - x^{n}}
\]
with $P(x) = \sum_{i = 0}^{n - 1}e_{i}x^{i}$, $e_{i} \in \mathbb{C}_{p}$.
Then,
\[
	\int_{\mathcal{A}(1,0),\varphi}f(x)\,dx = -\frac{1}{n}\sum_{\zeta^{n} = 1}\frac{\zeta P(\zeta)}{1 - \omega^{p^{\alpha}}(1 - \zeta)} =  -\frac{1}{n}\sum_{i = 0}^{n - 1}e_{i}\sum_{\zeta^{n} = 1}\frac{\zeta^{i + 1}}{1 - \omega^{p^{\alpha}}(1 - \zeta)}.
\]
\end{theorem}
\begin{proof}
The partial fraction decomposition of $f(x)$ is
\[
	 -\frac{1}{n}\sum_{\zeta^{n} = 1}\frac{\zeta P(\zeta)}{x - \zeta}.
\]
The theorem is a straightforward integration followed by the substitution of the expression for $P(x)$.
\end{proof}
Theorems \ref{T:ints same raylims}, \ref{T:raylims same} and \ref{T:intP/1-x^n} can be combined to yield:
\begin{theorem}\label{T:intwith raylims}
	Suppose $f\in H(\mathcal{A}(1,0))$ is bounded and $\varphi$-urc with $\varphi(k) = \alpha + \lambda k$. Suppose there's an $n \in \mathbb{Z}^{+}$, such that $p^{\lambda} \equiv 1 \pmod{n}$ and the ray limits $\mathcal{L}_{f,\varphi}(d)$ are constant on classes of $d \  \text{modulo}$\ $n$. Let $\alpha'$ be such that $\alpha'p^{\alpha} \equiv 1 \pmod{n}$.
Then,
\[
	\int_{\mathcal{A}(1,0), \varphi}f(x)\,dx = -\frac{1}{n}\sum_{i = 0}^{n - 1}\delta_{i}\mathcal{L}_{f,\varphi}(\alpha'(i + 1)),
\]
where,
\[
	\delta_{i} = \sum_{\zeta^{n} = 1}\frac{\zeta^{i + 1}}{1 - \omega^{p^\alpha}(1 - \zeta)}.
\]
The sum for $\delta_{i}$ is over the $n$-th roots of unity.
\end{theorem}
\begin{corollary}\label{C:intwith raylims}
	If $\alpha = 0$, then $\alpha' = 1$ and
\[
	\int_{\mathcal{A}(1,0), \varphi}f(x)\,dx = -\frac{1}{n}\sum_{i = 0}^{n - 1}\delta_{i}\mathcal{L}_{f,\varphi}(i + 1).
\]
\end{corollary}
Corollary \ref{C:intwith raylims} can be used to obtain
\begin{theorem}
	Let $j \in \mathbb{Z}^{+}$, $j \geq 2$. Let $\varphi(k) = k$ and $\psi_{p}(x)$ be the derivative of the p-adic log gamma function on $\mathbb{C}_{p}\setminus\mathbb{Z}_{p}$. \textup{(See Theorem \ref{T:intwithL*}.)} Then
\[
	\int_{\mathcal{A}(1,0), \varphi}x^{j - 3}\psi_{p}'(\frac{1}{x})\,dx = \frac{1 - j}{p - 1}\sum_{i = 0}^{p - 2}\delta_{i}L_{p}(j,\omega^{i - j + 2}),
\]
where
\[
	\delta_{i} = \sum_{\zeta^{p - 1 } = 1}\frac{\zeta^{i + 1}}{1 - \omega(1 - \zeta)}.
\]

\end{theorem}
\begin{proof}
When $|x|_{p} > 1$, the Laurent series for $\psi_{p}'(x)$ is (See \cite{hC2007}, \cite{jD1977} or \cite{nK1980}.)
\[
	\psi_{p}'(x) = \sum_{n \geq 1}\frac{B_{n - 1}}{x^{n}},
\]
where the $B_{n}$ are the Bernoulli numbers from $te^{t}/(e^{t} - 1)$.
Let
\[
	f(x) = x^{j - 3}\psi_{p}'(1/x) = \sum_{n \geq j - 2}B_{n - j + 2}x^{n}.
\]
From
\[
	B_{n} = -\left(1 - p^{n - 1}\right)^{-1}nL_{p}\left(1 - n, \omega^{n}\right)
\]
it follows that the ray limits at $-1$ are given by
\[
	\mathcal{L}_{f,\varphi}(d) = \lim_{k \to \infty}B_{dp^{k} - j + 1} = -(1 - j)L_{p}\left(j, \omega^{d - j + 1}\right).
\]
This limit is constant on classes $\pmod{(p - 1)}$. In a given class, the limit is uniform with respect to $d$. There are only $p - 1$ classes, so the limit is uniform with respect to $d \in \mathbb{Z}^{+}$. Hence $f(x)$ is $\varphi$-urc. Now Corollary \ref{C:intwith raylims} can be applied to obtain the integral.

\end{proof}

 \section{Cauchy's Theorems}
	L. G. Shnirelman, \cite{lS1938} and \cite{nK1980}, used his p-adic line integral to provide a p-adic version of Cauchy's Integral Formula. This result is similar to the complex version in that if $f$ is holomorphic on and inside a circle and $z$ is inside the circle, the theorem expresses the value of $f(z)$ as a line integral around the circle.\\
	
	In the complex plane we think of the boundary circle of a closed disc as a small part of the disc and the interior as the major part of the disc. In $\mathbb{C}_{p}$,  a boundary circle of a disc of radius $R$ is made up of countably many open discs of radius $R$ and the interior is just one such disc. \\
	
	In this sense a complex disc is mostly its interior and a p-adic disc is mostly on a boundary circle. From this point of view , the p-adic analog of Cauchy's Integral Formula is one that expresses the value of a function $f$ at points on a boundary circle in terms of the values interior to the circle or, equivalently, the value of $f$ at any point not on a particular arc in terms of the function values on that arc.\\
	
	Using a line integral on a small part of the circle, an arc, this section will have p-adic versions of the Residue Theorem, the Cauchy-Goursat Theorem  and Cauchy's Integral Formula.\\
	
	The domains considered will be closed discs and closed discs with a finite number of open discs removed. The functions will be Krasner analytic functions on these domains.\\

	Holes are permitted on a boundary circle and, depending on the theorem, open holes of radius $R$ in a disc of radius $R$ may be permitted.\\
	
 The first result is a version of the Residue Theorem for Laurent series converging on an annulus. It is a little broader than usual because the center point used to define the Laurent series and the annulus is not necessarily the same as the center point for the circle giving us an arc for integration.
\begin{theorem}[Residue Theorem for Laurent Series]\label{T:res Laur}
Let $\mathcal{C}$ be a circle with center $x_{0}$ and radius $R$. Suppose
\[
	f(x) = \sum_{n = -\infty}^{\infty}c_{n}(x - x_{0})^{n}\qquad \text{for}\ |x - x_{0}|_{p} = R.
\]
Let $b$ be a point on $\mathcal{C}$, $a$ be such that $|a - b|_{p} = R$ and $\alpha \in \mathbb{Z}_{\geq 0}$.
Then, 
\[
	\int_{\mathcal{A}(a,b),\Phi_{\alpha}}f(x)\,dx = \frac{c_{-1}}{1 - \omega^{p^{\alpha}}\left( \frac{a - x_{0}}{a - b}\right)}.
\]
\end{theorem}
\noindent $\omega$ is the Teichm\"{u}ller character extended so $\omega(x) = 0$ if $|x|_{p} < 1$.
\begin{proof}
$f(x)$ is the uniform limit of symmetric sums on the arc $\mathcal{A}(a,b)$. These sums are rational functions whose integrals are each 
\[
	\frac{c_{-1}}{1 - \omega^{p^{\alpha}}\left( \frac{a - x_{0}}{a - b}\right)}.
\]
Hence, by Theorem \ref{T:unif conv}, the theorem is proved.
\end{proof}
Notice that if a different $x_{0}$ in the interior of $\mathcal{C}$ is used for the expansion of $f(x)$, the value of $\omega((a - x_{0})/(a - b))$ is unchanged and hence $c_{-1}$, the \emph{residue of $f(x)$ at $x_{0}$} is unchanged.

Now let's look at broader domains than closed disks. The concept that will link values of a function is that of a Krasner analytic function, see~\cite{hC2007} and \cite{aR2000}. The domains to be considered will be closed discs with a finite number of open discs removed. A function is a Krasner analytic function on these domains if it is the uniform limit of rational functions whose poles are outside the domain.\\

\begin{theorem}[Residue Theorem]\label{T:residueKr}
Let
\[
	\mathcal{D} = D^{+}(a,R)\setminus\cup_{i = 1}^{M}D^{-}(x_{i},r_{i}),
\]
with $r_{i} < R$ for all $i$ and the $D^{+}(x_{i},r_{i})$ mutually disjoint. The $x_{i}$ do not have to be interior to the circle $|x - a|_{p} = R$. Choose $b$ so that $|b - a|_{p} = |b - x_{i}|_{p} = R$. Let $\alpha \in \mathbb{Z}_{\geq 0}$.\\
Suppose $f(x)$ is a Krasner analytic function  on $\mathcal{D}$.
Then
\[
	\int_{\mathcal{A}(a,b),\Phi_{\alpha}}f(x)\,dx = \sum_{i = 1}^{M}\frac{1}{1 - \omega^{p^{\alpha}}\left( \frac{a - x_{i}}{a - b}\right)}\underset{x = x_{i}}{Res}f(x)
\].
\end{theorem}
\begin{proof}
A Krasner analytic function on $\mathcal{D}$, see \cite{aR2000}, has the form
\[
	f(x) = \sum_{i = 0}^{M}f_{i}(x)
\]
with $f_{0}(x) \in H(D^{+}(a,R))$ and
\[
	f_{i}(x) = \sum_{n=1}^{\infty}\frac{c_{n,i}}{(x - x_{i})^{n}}, \qquad |x - x_{i}|_{p} 
	\geq r_{i}.
\]
Clearly,
\begin{equation}\label{E:resints}
	\int_{\mathcal{A}(a,b),\Phi_{\alpha}}f(x)\,dx = \sum_{i = 0}^{M}\int_{\mathcal{A}(a,b),\Phi_{\alpha}}f_{i}(x)\,dx.
\end{equation}
By Theorem \ref{T:holo=0},
\[
	\int_{\mathcal{A}(a,b),\Phi_{\alpha}}f_{0}(x)\,dx = 0,
\]
and by Theorem \ref{T:res Laur} we have
\[
	\int_{\mathcal{A}(a,b),\Phi_{\alpha}}f_{i}(x)\,dx = \frac{c_{1,i}}{1 - \omega^{p^{\alpha}}\left( \frac{a - x_{i}}{a - b}\right)} =\frac{\underset{x = x_{i}}{Res}f_{i}(x)}{1 - \omega^{p^{\alpha}}\left( \frac{a - x_{i}}{a - b}\right)}.
\]
The condition that the $D_{i}^{+}(x_{i},r_{i})$ are disjoint implies that $\underset{x = x_{i}}{Res}f_{i}(x) = \underset{x = x _{i}}{Res}f(x)$.
Substitution into Equation \eqref{E:resints} completes the proof of the Residue Theorem.
\end{proof}

In the complex case, the Cauchy-Goursat Theorem for analytic functions on a multiply connected  region says, roughly, that if a function is analytic on disc with holes, the integral of the function around the perimeter of the disc is equal to the sum of the integrals taken around the holes. In the complex case, the holes must have smaller radii than the disc. This isn't so in the p-adic case, so two theorems will be given.

\begin{theorem}[Cauchy-Goursat Theorem - small holes]\label{T:Cauchy-G1}
Let $\mathcal{C}$ be a circle with center $a$ and radius $R$. Let $\mathcal{D} = D^{+}(a,R)\setminus\cup_{i = 1}^{M}D^{-}(x_{i},\rho_{i})$, with all $\rho_{i} < R$. Let the $\mathcal{C}_{i}$ be circles with centers $x_{i}$ and radii $r_{i}$ with $\rho_{i} \leq r_{i} < R$. Let $b$ be a point on $\mathcal{C}$ with $|b - x_{i}|_{p} = R$ for all $i$ and $b_{i}$ be a point on $\mathcal{C}_{i}$. Suppose that the $D^{+}(x_{i},r_{i})$ are mutually disjoint. Let $\alpha \in \mathbb{Z}_{\geq 0}$.\\
Then, if $f(x)$ is a Krasner analytic function on  $\mathcal{D}$, 
\[
	\int_{\mathcal{A}(a,b),\Phi_{\alpha}}f(x)\,dx = \sum_{i = 1}^{M}\frac{1}{1 - \omega^{p^{\alpha}}\left( \frac{a - x_{i}}{a - b}\right)}\int_{\mathcal{A}(x_{i},b_{i}),\Phi_{\alpha}}f(x)\,dx.
\]
The constant coefficients depend only on $\alpha$, the arc $\mathcal{A}(a,b)$, $\mathcal{D}$ and the choice of boundary circle $\mathcal{C}$.
\end{theorem}
\begin{proof}
We can write
\[
	f(x) = \sum_{i = 0}^{M}f_{i}(x)
\]
with $f_{0}(x) \in H(D^{+}(a,R))$ and
\[
	f_{i}(x) = \sum_{n=1}^{\infty}\frac{c_{n,i}}{(x - x_{i})^{n}}, \qquad |x - x_{i}|_{p} 
	\geq r_{i}.
\]
First, let's integrate $f_{i}(x)$ over the arc $\mathcal{A}(x_{i},b_{i})$. By Theorem \ref{T:res Laur} we have
\[
	\int_{\mathcal{A}(x_{i},b_{i}),\Phi_{\alpha}}f_{i}(x)\,dx = \frac{c_{1,i}}{1 - \omega^{p^{\alpha}}\left( \frac{x_{i} - x_{i}}{x_{i} - b_{i}}\right)} = c_{1,i}.
\]
Substitution into the Residue Theorem, Theorem \ref{T:residueKr}, yields
\begin{equation}\label{E:CauchyG1}
	\int_{\mathcal{A}(a,b),\Phi_{\alpha}}f(x)\,dx = \sum_{i = 1}^{M}\frac{1}{1 - \omega^{p^{\alpha}}\left( \frac{a - x_{i}}{a - b}\right)}\int_{\mathcal{A}(x_{i},b_{i}),\Phi_{\alpha}}f_{i}(x)\,dx.
\end{equation}
Now let's integrate $f_{j}(x)$, $j \neq i$, over the arc $\mathcal{A}(x_{i},b_{i})$
The assumption that the $D_{i}^{+}(x_{i},r_{i})$ are mutually disjoint implies that $f_{j}(x)$ is holomorphic on $|x - x_{i}|_{p} \leq r_{i}$. It follows that
\[
	\int_{\mathcal{A}(x_{i},b_{i}),\Phi_{\alpha}}f_{j}(x)\,dx = 0, \qquad j \neq i.
\]
Hence
\begin{equation}\label{E:CauchyG2}
	\int_{\mathcal{A}(x_{i},b_{i}),\Phi_{\alpha}}f(x)\,dx = \int_{\mathcal{A}(x_{i},b_{i}),\Phi_{\alpha}}f_{i}(x)\,dx.
\end{equation}

Substituting equation \eqref{E:CauchyG2} into equation \eqref{E:CauchyG1} yields this version of the Cauchy-Goursat Theorem.
\end{proof}
If all of  the $x_{i}$, are interior to the circle, then the coefficients in this theorem become $1$ and, as in the complex case, we have
\[
	\int_{\mathcal{A}(a,b),\Phi_{\alpha}}f(x)\,dx = \sum_{i = 1}^{M}\int_{\mathcal{A}(x_{i},b_{i}),\Phi_{\alpha}}f(x)\,dx.
\]
It's helpful to establish the following lemma and theorem before considering domains with holes of radius $R$.
\begin{lemma}\label{L:D id}
Let $x_{i}$ and $b_{i}$, $1 \leq i\leq n$, be formal variables. Then 
\[
	det\left(\frac{x_{i} - b_{i}}{x_{j} - b_{i}}\right)_{1 \leq i,j \leq n} = \prod_{\substack{i,j = 1\\ i \neq j}} ^{n} \left(x_{j} - b_{i}\right)^{-1}\prod_{\substack{i,j = 1\\i < j}} ^{n}\left(b_{j} - b_{i}\right)\left(x_{i} - x_{j}\right).
\] 
\end{lemma}
\begin{proof}
\begin{align*}
	 det\left(\frac{x_{i} - b_{i}}{x_{j} - b_{i}}\right)_{1 \leq i,j \leq n} \cdot \prod_{\substack{i,j = 1\\ i \neq j}} ^{n} \left(x_{j} - b_{i}\right) &= det\left(\frac{\prod_{k = 1}^{n}(x_{k} - b_{i})}{x_{j} - b_{i}}\right)_{1 \leq i,j \leq n}\\
& = \prod_{\substack{i,j = 1\\i < j}} ^{n}\left(b_{j} - b_{i}\right)\left(x_{i} - x_{j}\right).
\end{align*}
The last step is derived from the facts that both sides of the last equality are polynomials, the linear factors on the right are factors on the left, the degrees of the variables on the left don't exceed their degrees on the right (This equates the two sides up to a multiplicative constant.) and, finally, that letting $x_{i} = b_{i}$, for all $i$, shows the multiplicative constant is $1$. 
\end{proof}
\begin{theorem}\label{T:D=1}
Suppose $x_{i}$, $b_{i} \in D^{+}(a,R)$, $1 \leq i \leq n$, $|x_{i} - x_{j}|_{p} = |b_{i} - b_{j}|_{p} = R$, if $i \neq j$, and $|x_{i} - b_{j}|_{p} = R$ for all $i$, $j$. Let $\alpha \in \mathbb{Z}_{\geq 0}$. Define the matrix $\textbf{D}_{\alpha}$ by

\[
	\textbf{D}_{\alpha} = \left(\frac{1}{1 - \omega^{p^{\alpha}}\left(\frac{x_{i} - x_{j}}{x_{i} - b_{i}}\right)}\right)_{1 \leq i,j \leq n}.
\]
Then,
\[
	\left|det(\textbf{D}_{\alpha})\right|_{p} = 1.
\]
\end{theorem}
\begin{proof}
The proof is by induction on $\alpha$, starting with $\alpha = 0$.
Since we can multiply all the $x_{i}$ and $b_{i}$ by a single non-zero constant with no change in $\omega((x_{i} - x_{j})/(x_{i} - b_{i}))$, we can take $R = 1$ with no loss.
Since $|\omega(x) - x|_{p} < 1$ for any $x$ with $|x|_{p} \leq 1$, it follows that
\[
	\frac{1}{1 - \omega\left(\frac{x_{i} - x_{j}}{x_{i} - b_{i}}\right)} = \frac{x_{i} - b_{i}}{x_{j} - b_{i}} + \varepsilon_{i,j}, \qquad \text{with}\  |\varepsilon_{i,j}|_{p} < 1.
\]
Each of the terms in $\textbf{D}_{\alpha}$ has absolute value $1$, so we can say
\[
	\left|det(\textbf{D}_{\alpha}) \right|_{p} = \left| det\left(\frac{x_{i} - b_{i}}{x_{j} - b_{i}}\right)_{1 \leq i,j \leq n} + \varepsilon \right|_{p}, \qquad \text{with}\  |\varepsilon|_{p} < 1.
\]
The conditions of this theorem, together with Lemma \ref{L:D id}, show that
\[
	 \left|det(\textbf{D}_{\alpha}) \right|_{p} = 1
\]
in the case $\alpha = 0$.\\
The induction is straightforward using 
\[
	(x + y)^{p} = x^{p} + y^{p} + \varepsilon,
\]
where $|\varepsilon|_{p}| \leq 1/p$ and $|x|_{p} = |y|_{p} = 1$.
\end{proof}

Now we're ready to establish a p-adic version of the Cauchy-Goursat Theorem that will allow for holes of radius $R$ in a closed disc of radius $R$.
\begin{theorem}[Cauchy-Goursat Theorem]\label{T:Cauchy-G2}
Let $\mathcal{C}$ be a circle with center $a$ and radius $R$. Let $\mathcal{D} = D^{+}(a,R)\setminus\cup_{i = 1}^{M}D^{-}(x_{i},\rho_{i})$, with $\rho_{i} = R$ for $i = 1,\dots,m$ and $\rho_{i} < R$ for $m <i \leq M$. Suppose that $|a - x_{i}|_{p} = R$ for $i \leq m$ and that the $D^{-}(x_{i},\rho_{i})$ are mutually disjoint. Let $b$ be a point on $\mathcal{C}$ with $|b - x_{i}|_{p} = R$ for all $i$. Let the $\mathcal{C}_{i}$ be circles having centers $x_{i}$ and radii $r_{i}$ with $\rho_{i} \leq r_{i} \leq R$. Suppose the $D^{+}(x_{i},r_{i})$ are mutually disjoint for $i > m$.  For each $i$, let $b_{i}$ be a point on $\mathcal{C}_{i}$ with the conditions, for $1 \leq i \leq m$, that $|b_{i} - x_{j}|_{p} = R$ for any $j$ and $|b_{i} - b_{j}|_{p} = R$ if $i \neq j$. Let $\alpha \in \mathbb{Z}_{\geq 0}$.\\
 
Then there are constants $\mu_{1}, \dotsc, \mu_{M}$, depending only on $\alpha$, the region $\mathcal{D}$ with its choice of boundary circle and the arcs $\mathcal{A}(a,b)$ and $\mathcal{A}(x_{i},b_{i})$, such that for any function $f(x)$ that is Krasner analytic on $\mathcal{D}$, we have
\[
	\int_{\mathcal{A}(a,b),\Phi_{\alpha}}f(x)\,dx = \sum_{i = 1}^{M}\mu_{i} \int_{\mathcal{A}(x_{i},b_{i}),\Phi_{\alpha}}f(x)\,dx.
\]
\end{theorem}
\begin{proof}
We can write
\[
	f(x) = \sum_{i = 0}^{M}f_{i}(x)
\]
with $f_{0}(x) \in H(D^{+}(a,R))$ and
\[
	f_{i}(x) = \sum_{n=1}^{\infty}\frac{c_{n,i}}{(x - x_{i})^{n}}, \qquad |x - x_{i}|_{p} 
	\geq r_{i}.
\]
Hence,
\begin{equation}\label{E:Cauchy-G3}
	\int_{\mathcal{A}(a,b),\Phi_{\alpha}}f(x)\,dx = \sum_{i = 0}^{M}\int_{\mathcal{A}(a,b),\Phi_{\alpha}}f_{i}(x)\,dx = \sum_{i = 1}^{M}\frac{c_{1,i}}{1 - \omega^{p^{\alpha}}\left( \frac{a - x_{i}}{a - b}\right)}.
\end{equation}
For $1 \leq i \leq m$, we have
\[
\int_{\mathcal{A}(x_{i},b_{i}),\Phi_{\alpha}}f(x)\,dx = \sum_{j = 0}^{M}\int_{\mathcal{A}(x_{i},b_{i}),\Phi_{\alpha}}f_{j}(x)\,dx = \sum_{j = 1}^{M}\frac{c_{1,j}}{1 - \omega^{p^{\alpha}}\left( \frac{x_{i} - x_{j}}{x_{i} - b_{i}}\right)}.
\]
These are $m$ linear equations in the $M$ variables $c_{1,j}$, $j = 1,\dotsc, M$.\\
For $m + 1 \leq i \leq M$ we have, as in the proof of the preceding theorem,
\[
	\int_{\mathcal{A}(x_{i},b_{i}),\Phi_{\alpha}}f(x)\,dx = \int_{\mathcal{A}(x_{i},b_{i}),\Phi_{\alpha}}f_{i}(x)\,dx = c_{1,i}.
\]
These are $M - m$ additional simple linear equations in $c_{1,j}$, $j = 1,\dotsc,M$.
Theorem \ref{T:D=1} and the diagonal nature of the last $M - m$ equations show that the $M$ equations are linearly independent. By Cramer's rule for solving linear equations, the $c_{1,j}$ are each linear combinations of the integrals
\[
\int_{\mathcal{A}(x_{i},b_{i}),\Phi_{\alpha}}f(x)\,dx. 
\]
Substitution of these expressions for the $c_{1,j}$ into equation \eqref{E:Cauchy-G3} completes the proof.
\end{proof}
Three versions of Cauchy's Integral Formula, corresponding to increasingly general domains, will be given. A fourth version, with a slightly different approach in a special case, will follow.\\

The first version, for functions holomorphic on a closed disc, shows how knowledge of a function on a single arc determines the function values anywhere else on the disc. 
\begin{theorem}[Cauchy's Integral Formula on closed discs]\label{T:Cauchy disc}
Let $D = D^{+}(a,R)$. Let $f \in H(D)$. Choose $b$ with $|b - a|_{p} = R$ and $z$ with $z \in D\setminus\mathcal{A}(a$, $b)$. Let $\alpha \in \mathbb{Z}_{\geq 0}$.
Then,
\[
	f(z) = \left(1 - \omega^{p^{\alpha}}\left( \frac{a - z}{a - b}\right)\right)\int_{\mathcal{A}(a,b),\Phi_{\alpha}}\frac{f(x)}{x - z}\,dx.
\]
\end{theorem}
The expression in front of the integral sign is constant for $z$ within an open disc of radius $R$.
 \begin{proof}
 The Laurent series for $f(x)/(x - z)$, expanded around $z$, and Theorem \ref{T:res Laur} give the result immediately.
 \end{proof}
The same method provides Cauchy's Formula for the derivatives of a holomorphic function on a closed disc.
\begin{theorem}
Let $D = D^{+}(a,R)$. Let $f \in H(D)$. Choose $b$ with $|b - a|_{p} = R$ and $z$ with $z \in D\setminus\mathcal{A}(a$, $b)$. Let $\alpha \in \mathbb{Z}_{\geq 0}$.
Then,
\[
	\frac{f^{(n)}(z)}{n!} = \left(1 - \omega^{p^{\alpha}}\left( \frac{a - z}{a - b}\right)\right)\int_{\mathcal{A}(a,b),\Phi_{\alpha}}\frac{f(x)}{(x - z)^{(n + 1)}}\,dx.
\].
\end{theorem}
The expression in front of the integral sign is constant for $z$ within an open disc of radius $R$.\\
The following known result follows easily from Theorem \ref{T:Cauchy disc}.
\begin{theorem}\label{T:maxmod}
Let $D = D^{+}(a,R)$.  If $f \in H(D)$, then 
\[
\sup_{x \in \mathcal{A}}|f(x)|_{p} = \max_{x \in D}|f(x|_{p}
\]
for every open disc $\mathcal{A}$ of radius $R$.
\end{theorem}
\begin{proof}
Theorem \ref{T:max int} applied to the integral in Cauchy's Theorem yields
\[
	\max_{x \in D\backslash \mathcal{A}}|f(x|_{p} \leq \sup_{x \in \mathcal{A}}|f(x)|_{p}
\]
for every arc $\mathcal{A}$. Hence,
\[
	\sup_{x \in \mathcal{A}}|f(x)|_{p} = \max_{x \in D}|f(x|_{p}.
\]
\end{proof}

This next version of Cauchy's Intergral Theorem is for a disc with a finite number of holes, all of whose radii are less than the radius of the disc. This result differs from the complex variables theorem in that the holes can be on the boundary circle and $z$ can be on the boundary circle. Any circle of radius $R$ can be used as the boundary circle of the disc. As in the complex case we use integrals on the boundary circle of the disc and the boundary circles of the holes. These integrals will be on arcs of these circles.
\begin{theorem}[Cauchy's Integral Formula - small holes]\label{T:Cauchy intK small}
Let $\mathcal{C}$ be a circle with center $a$ and radius $R$. Let $\mathcal{D} = D^{+}(a,R)\setminus \cup_{i = 1}^{M}D^{-}(x_{i},r_{i})$ with $r_{i} < R$ and the $D^{+}(x_{i},r_{i})$ mutually disjoint. Let $b$ be a point on $\mathcal{C}$ with $|b - x_{i}|_{p} = R$ for all $i$. Choose $b_{i}$, $i = 1, \dots, M$, so that $|x_{i} - b_{i}|_{p} = r_{i}$. Let $x_{0} \in \mathcal{D}$ be such that $|x_{0} - x_{i}|_{p} > r_{i}$ and $|x_{0} - b|_{p} = R$.
Let $r^{*} = \min_{1 \leq i \leq M}|x_{0} - x_{i}|_{p}$ and $\mathcal{D}^{*} = D^{-}(x_{0}, r^{*})$.
Let $\alpha \in \mathbb{Z}_{\geq 0}$.\\
Then, if $z \in \mathcal{D}^{*}$ and $f(x)$ is a Krasner analytic function on $\mathcal{D}$,
\[
	\frac{f(z)}{1 - \omega^{p^{\alpha}}\left( \frac{a - x_{0}}{a - b}\right)}  = \int_{\mathcal{A}(a,b),\Phi_{\alpha}}\frac{f(x)}{x - z}\,dx - \sum_{i = 1}^{M}\frac{1}{1 - \omega^{p^{\alpha}}\left( \frac{a - x_{i}}{a - b}\right)}\int_{\mathcal{A}(x_{i},b_{i}),\Phi_{\alpha}}\frac{f(x)}{x - z}\,dx.
\]
The constant coefficients depend only on $\alpha$, the arc $\mathcal{A}(a,b)$, $\mathcal{D}^{*}$, $\mathcal{D}$ and the choice of boundary circle $\mathcal{C}$.
\end{theorem}
\begin{proof}
Choose $r_{0}$ so $|z - x_{0}|_{p} < r_{0} < r^{*}$. Now select $b_{0}$ so $|x_{0} - b_{0}|_{p} = r_{0}$. Let $\mathcal{D}' = \mathcal{D} \setminus D^{-}(x_{0},r_{0})$.
The function $f(x)/(x - z)$ on $\mathcal{D}'$ satisfies the conditions of Theorem \ref{T:Cauchy-G1}, the \emph{Cauchy-Goursat Theorem - small holes}, so we have
\begin{equation}\label{E:Cauchy intK small}
	\int_{\mathcal{A}(a,b),\Phi_{\alpha}}\frac{f(x)}{x - z}\,dx = \sum_{i = 0}^{M}\frac{1}{1 - \omega^{p^{\alpha}}\left( \frac{a - x_{i}}{a - b}\right)}\int_{\mathcal{A}(x_{i},b_{i}),\Phi_{\alpha}}\frac{f(x)}{x - z}\,dx.
\end{equation}
Cauchy's Theorem for functions holomorphic on a closed disc, Theorem \ref{T:Cauchy disc}, yields
\[
	f(z) = \int_{\mathcal{A}(x_{0},b_{0}),\Phi_{\alpha}}\frac{f(x)}{x - z}\,dx.
\]
Substitution of $f(z)$ for this integral in equation \eqref{E:Cauchy intK small} and rearranging terms yields
\[
	\frac{f(z)}{1 - \omega^{p^{\alpha}}\left( \frac{a - x_{0}}{a - b}\right)}  = \int_{\mathcal{A}(a,b),\Phi_{\alpha}}\frac{f(x)}{x - z}\,dx - \sum_{i = 1}^{M}\frac{1}{1 - \omega^{p^{\alpha}}\left( \frac{a - x_{i}}{a - b}\right)}\int_{\mathcal{A}(x_{i},b_{i}),\Phi_{\alpha}}\frac{f(x)}{x - z}\,dx. 
\]
\end{proof}
We can observe that the value of $\omega\left( \frac{a - x_{i}}{a - b}\right)$  is constant for values of $x_{i}$ within an open disc of radius $R$.\\

This next result, the most general, is Cauchy's Integral Formula for Krasner analytic functions on a disc with a finite number of open holes of any size. It will be assumed that at least one hole has radius $R$. $z$ can be anywhere in the domain except on the boundary circle of a small hole. This situation does not have a complex analog. The boundary circles of holes of radius $R$ are themselves boundary circles for the closed disc. Hence we distinguish a boundary circle for the disc separate from the boundary circles of the holes.
 \begin{theorem}[Cauchy's Integral Formula]\label{T:Cauchy intK}
Let $\mathcal{C}$ be a circle with center $a$ and radius $R$. Let $\mathcal{D} = D^{+}(a,R)\setminus\cup_{i = 1}^{M}D^{-}(x_{i},r_{i})$, with $r_{i} = R$ for $i = 1,\dots,m$ and $r_{i} < R$ for $m <i \leq M$. Suppose that $|a - x_{i}|_{p} = R$ for $i \leq m$, the $D^{-}(x_{i},r_{i})$ are mutually disjoint and the $D^{+}(x_{i},r_{i})$ are mutually disjoint for $i > m$. Let $b$ be a point on $\mathcal{C}$ with $|b - x_{i}|_{p} = R$ for all $i$. Let the $\mathcal{C}_{i}$ be circles having centers $x_{i}$ and radii $r_{i}$. For each $i$, let $b_{i}$ be a point on $\mathcal{C}_{i}$ with the condition, for $1 \leq i \leq m$, $|b_{i} - x_{j}|_{p} = R$ for any $j$ and $|b_{i} - b_{j}|_{p} = R$ if $i \neq j$. Let $\alpha \in \mathbb{Z}_{\geq 0}$.\\
Let $\bar{x} \in \mathcal{D}$ be such that
\[
	|\bar{x} - b_{i}|_{p} = R, \  \textup{for} \ 1 \leq i \leq m, \quad \textup{and}\quad |\bar{x} - x_{i}|_{p} > r_{i}, \  \textup{for} \ i > m.
\]
Let
\[r^{*} = \min_{1 \leq i \leq M}|\bar{x} - x_{i}|_{p} \quad \textup{and} \quad \mathcal{D}^{*} = D^{-}(\bar{x}, r^{*}).
\]
Then, if $z \in \mathcal{D}^{*}$, there are constants $\mu_{0}$, $\mu_{1}, \dotsc, \mu_{M}$, depending only on $\alpha$, $\mathcal{D}^{*}$, the region $\mathcal{D}$ with its choice of boundary circle, the arcs $\mathcal{A}(a, b)$ and $\mathcal{A}(x_{i}, b_{i})$, so that
\[
	f(z) = \mu_{0}\int_{\mathcal{A}(a,b),\Phi_{\alpha}}\frac{f(x)}{x - z}\,dx +  \sum_{i =1}^{M}\mu_{i}\int_{\mathcal{A}(x_{i},b_{i}),\Phi_{\alpha}}\frac{f(x)}{x - z}\,dx.
\]
\end{theorem}
\begin{proof}
It's convenient to write $x_{o} = a$ and $b_{0} = b$.
Let $z \in \mathcal{D}^{*}$. Let $\mathcal{D}' = \mathcal{D}\setminus{D^{*}}$.
The function $f(x)/(x - z)$ is a Krasner analytic function on  $\mathcal{D}'$. Hence it can be written
\[
	\frac{f(x)}{x - z}  = \frac{f(z)}{x - z} + \sum_{j = 0}^{M}f_{j}(x)
\]
with $f_{0}(x) \in H(D^{+}(a,R))$ and
\[
	f_{j}(x) = \sum_{n=1}^{\infty}\frac{c_{n,j}}{(x - x_{j})^{n}}, \qquad \text{for}\ |x - x_{j}|_{p} 
	\geq r_{j}\ \text{and}\  j \geq 1.
\]

Now let's integrate  $f(x)/(x - z)$ on each of the arcs $\mathcal{A}(x_{i},b_{i})$.
If $0 \leq i \leq m$ we have
\[
	\int_{\mathcal{A}(x_{i},b_{i}),\Phi_{\alpha}}\frac{f(x)}{x - z}\,dx = \frac{f(z)}{1 - \omega^{p^{\alpha}}\left( \frac{x_{i} - z}{x_{i} - b_{i}}\right)} + \sum_{j=1}^{M}\frac{c_{1,j}}{1 - \omega^{p^{\alpha}}\left( \frac{x_{i} - x_{j}}{x_{i} - b_{i}}\right)}.
\]
For $m < i \leq M$ we have, because $|x_{j} - x_{i}|_{p} > |r_{i}|_{p}$ when $j \neq i$.
\[
	\int_{\mathcal{A}(x_{i},b_{i}),\Phi_{\alpha}}\frac{f(x)}{x - z}\,dx = c_{1,i}.
\]
Because $|z - \bar{x}|_{p} < R$,  we can replace $z$ by $\bar{x}$, when $i \leq m$, in the expression
\[
	\omega\left( \frac{x_{i} - z}{x_{i} - b_{i}}\right).
\]
We now have $M + 1$ linear equations, in the $M + 1$ variables $f(z)$, $c_{1,1}, \dotsc, c_{1,M}$, described as follows:
The first $m + 1$ rows of the coefficient matrix consist of a first column 
\[
	\left(\frac{1}{1 - \omega^{p^{\alpha}}\left( \frac{x_{i} - \bar{x}}{x_{i} - b_{i}}\right)}\right)_{0 \leq i \leq m}.
\]
next to the matrix
\[
	\left(\frac{1}{1 - \omega^{p^{\alpha}}\left( \frac{x_{i} - x_{j}}{x_{i} - b_{i}}\right)}\right)_{0 \leq i \leq m, 1 \leq j \leq M}.
\]
The remaining rows have $1$ on the main diagonal and $0$ everywhere else.\\
The separate column of constants is
\[
\left(\int_{\mathcal{A}(x_{i},b_{i}),\Phi_{\alpha}}\frac{f(x)}{x - z}\,dx\right)_{0 \leq i \leq M}.
\]
The determinant of the coefficient matrix of these equations is not zero. To see this, multiply the top row of the matrix by
\[
	\left(\frac{\bar{x} - b_{0}}{x_{0} - b_{0}}\right)^{p^{\alpha}}.
\]
This factor is not zero. To within an error of less than $1$ in each term, the top row becomes, after the first element $1$,
\[
	\left(\frac{1}{1 - \omega^{p^{\alpha}}\left( \frac{\bar{x} - x_{j}}{\bar{x} - b_{0}}\right)}\right)_{1 \leq j \leq M}.
\]	
This gives a determinant to which Theorem \ref{T:D=1} can be applied to conclude the coefficient matrix has a non-zero determinant. Hence Cramer's rule can be used to solve the equations for $f(z)$. The theorem follows.  
\end{proof}

Here's an example with $p = 5$. Let $\mathcal{D} \subset \mathbb{C}_{5}$ be the closed unit disc around $0$ with the open unit disc around $0$ removed. We'll take the circle with center $-1$ and radius $1$ as the boundary circle for the disc. Suppose $f(x)$ is a Krasner analytic function on $\mathcal{D}$. We'll find an expression for $f(z)$ when $z$ is in the interior of the boundary circle,  the open unit disc around $-1$. We'll assume that the values of $f(x)$ are known on the arcs $\mathcal{A}(-1,1)$ and $\mathcal{A}(0,2)$. Let $\bar{x} = -1$. The first arc is on the boundary circle of the disc and the second on the boundary circle of the hole.

Take $\alpha = 0$, $x_{0} = -1$, $x_{1} = 0$, $b_{0} = 1$ and $b_{1} = 2$. In $\mathbb{C}_{5}$ we can write $\omega(2) = i$. Putting these numbers into the proof of Theorem \ref{T:Cauchy intK} yields:\\
If $|z + 1|_{p} < 1$,
\[
	f(z) = 2\int_{\mathcal{A}(-1,1),\Phi_{0}}\frac{f(x)}{x - z}\,dx -(1 - i)\int_{\mathcal{A}(-0,2),\Phi_{0}}\frac{f(x)}{x - z}\,dx.
\]\\

The next version of Cauchy's Integral Formula applies to a closed disc of radius $R$ with one open hole of radius $R$. It's different from Theorem \ref{T:Cauchy intK} in that the two circles used for the integrals use the same base point for an arc. That is,  $b_{0} = b_{1}$, which is impossible in Theorem \ref{T:Cauchy intK}. This means that while there will still be two integrals, knowledge of the values of $f(x)$ on only one arc is needed to obtain values of $f(z)$ for $z$ elsewhere in the domain. This approach also requires that a determinant $\textbf{D} \neq 0$. This is trickier to establish than the similar requirement in Theorem \ref{T:Cauchy intK}.

 \begin{theorem}[Cauchy's Integral Formula -one large hole ]\label{T:Cauchy intK1hole}
Let $\mathcal{D} = D^{+}(x_{1}, R)\backslash D^{-}(x_{1}, R)$. Select $b \in \mathcal{D}$ so that $|x_{1} - b|_{p} = R$. Let $x_{0} \in \mathcal{D}$ with $|x_{0} - b|_{p} = R$. Choose $a_{0} \in D^{+}(x_{1},R)$ so that $|a_{0} - b|_{p} = |a_{0} - x_{1}|_{p} = R$. For ease of notation, let $a_{1} = x_{1}$. Let $f$ be a Krasner analytic function on $\mathcal{D}$.\\
Then, if $|z - x_{0}|_{p} < R$ and $\textbf{D} \neq 0$,
\[
	f(z) = \frac{1}{\textbf{D}}\left(\int_{\mathcal{A}(a_{0},b),\Phi_{\alpha}}\frac{f(x)}{x - z}\,dx -  \frac{1}{1 - \omega^{p^{\alpha}}\left( \frac{a_{0} - x_{1}}{a_{0} - b}\right)}\int_{\mathcal{A}(x_{1},b),\Phi_{\alpha}}\frac{f(x)}{x - z}\,dx\right),
\]
where
\[
\textbf{D} = \det \left(\frac{1}{1 - \omega^{p^{\alpha}}\left( \frac{a_{i} - x_{j}}{a_{i} - b}\right)}\right)_{0 \leq i,j \leq 1}.
\]
Furthermore, $a_{0}$ can always be chosen so that $\textbf{D} \neq 0$.
\end{theorem}
\begin{proof}
$f(x)/(x - z)$ is Krasner analytic on the arcs $\mathcal{A}(a_{i},b)$, so the integrals over these arcs will exist. Also, $f(x)/(x - z)$ is Krasner analytic on the domain
$\mathcal{D}\backslash D^{-}(z,R) = \mathcal{D}\backslash D^{-}(x_{0},R)$.
Hence,
\[
	\frac{f(x)}{x - z} = f_{0}(x) + f_{1}(x) + f_{2}(x),
\]
with $f_{0}(x) \in H(D^{+}(x_{1},R))$,
\[
	f_{1}(x) = \frac{f(z)}{x - z} \qquad \text{and} \quad f_{2}(x) = \sum_{n \geq 1}\frac{c_{n}}{(x - x_{1})^{n}}.
\]
Since
\[\omega\left(\frac{a_{i} - z}{a_{i} - b}\right) = \omega\left( \frac{a_{i} - x_{0}}{a_{i} - b}\right),
\]
we have the two equations
\[
	\int_{\mathcal{A}(a_{i},b),\Phi_{\alpha}}\frac{f(x)}{x - z}\,dx = \frac{f(z)}{1 - \omega^{p^{\alpha}}\left( \frac{a_{i} - x_{0}}{a_{i} - b}\right)} + \frac{c_{1}}{1 - \omega^{p^{\alpha}}\left( \frac{a_{i} - x_{1}}{a_{i} - b}\right)}\quad i = 1,2.
\]
These are linear equation in $f(z)$ and $c_{1}$, so Cramer's Rule applies and, if $\textbf{D} \neq 0$, we can solve for $f(z)$ as in the statement of the theorem.
\end{proof}
The following result shows that for any $f$ and $z$ that satisfy the conditions of Theorem \ref{T:Cauchy intK1hole} we can select $a_{0}$ so that $\textbf{D} \neq 0$.
\begin{theorem}\label{T:Cauchya0a1}
Let $\zeta = \omega\left(\frac{x_{1} - x_{0}}{x_{1} - b}\right)$. If $\zeta$ is not a primitive  $6$-th root of unity, then if $a_{0} = x_{0}$ we have $\textbf{D} \neq 0$. If  $\zeta$ is a primitive  $6$-th root of unity, then any $a_{0}$ such that $\omega\left(\frac{a_{0} - x_{1}}{a_{0} - b}\right) \neq -\zeta$ will have $\textbf{D} \neq 0$.
\end{theorem}
\begin{proof}
If $a_{0} = x_{0}$ and $\omega\left(\frac{x_{1} - x_{0}}{x_{1} - b}\right)$ is not a primitive $6$-th root of unity, then
\begin{equation*}
\textbf{D} =
\left|
\begin{matrix}
	 1&\frac{1}{1 - \omega^{p^{\alpha}}\left( \frac{x_{0} - x_{1}}{x_{0} - b}\right)}\\\frac{1}{1 - \omega^{p^{\alpha}}\left( \frac{x_{1} - x_{0}}{x_{1} - b}\right)}&1
\end{matrix}
\right|.
\end{equation*}
Let $\zeta_{1}$ and $\zeta_{2}$ be the two roots of unity expressed by $\omega^{p^{\alpha}}$. If $\textbf{D} = 0$, then we have
\[
	(1 - \zeta_{1})(1 - \zeta_{2}) = 1.
\]
This is equivalent to 
\[
	1 + \frac{\zeta_{1}}{\zeta_{2}} = \zeta_{1}.
\]
A look at the complex circle shows $\zeta_{1}$ and $\zeta_{2}$ are the two primitive $6$-th roots of unity.\\
Now for the case when $\zeta = \omega\left(\frac{x_{1} - x_{0}}{x_{1} - b}\right)$ is a primitive $6$-th root of unity. We have
\begin{equation*}
\textbf{D} =
\left|
\begin{matrix}
	 \frac{1}{1 - \omega^{p^{\alpha}}\left( \frac{a_{0} - x_{0}}{a_{0} - b}\right)}&\frac{1}{1 - \omega^{p^{\alpha}}\left( \frac{a_{0} - x_{1}}{a_{0} - b}\right)}\\\frac{1}{1 - \zeta^{(p^{\alpha}}}&1
\end{matrix}
\right|.
\end{equation*}
Let $\zeta_{1}$, $\zeta_{2}$, $\zeta_{3}$ be the three roots of unity expressed by $\omega^{p^{\alpha}}$ as below. Then
\begin{equation*}
\textbf{D} =
\left|
\begin{matrix}
	 \frac{1}{1 - \zeta_{1}}&\frac{1}{1 - \zeta_{2}}\\ \frac{1}{1 - \zeta_{3}}&1
\end{matrix}
\right|.
\end{equation*}
In the complex plane, the mapping $z \to 1/(1 - z)$ sends the unit circle onto the extended line $\Re z = 1/2$. It follows that there is at most one solution, $(\zeta_{1}$, $\zeta_{2})$, to the equation $\textbf{D} = 0$. The pair $(\zeta_{1}$, $\zeta_{2}) = (\zeta_{3}^{2}$, $-\zeta_{3})$ is that solution. From $\zeta_{2} = - \zeta_{3}$ we get
\[
	\omega\left(\frac{a_{0} - x_{1}}{a_{0} - b}\right) = -\zeta.
\]
\end{proof}
Let's look again at the example that followed Theorem \ref{T:Cauchy intK}. 
$f(x)$ is a Krasner analytic function on $D^{+}(0,1)\setminus D^{-}(0,1)$ in $\mathbb{C}_{5}$. We'll assume we know the values of $f(x)$ when $|x - 1|_{p} < 1$. Then for $|z + 1|_{p} < 1$, using Theorem \ref{T:Cauchy intK1hole} with $\alpha = 0$, $x_{0} = -1$, $x_{1} = 0$, $b = 1$, we have
\[
	f(z) = \frac{6 - 2i}{5}\int_{\mathcal{A}(-1,1),\Phi_{0}}\frac{f(x)}{x - z}\,dx - \frac{2 - 4i}{5}\int_{\mathcal{A}(0,1),\Phi_{0}}\frac{f(x)}{x - z}\,dx. 
\]

The fact that $|\textbf{D}|_{5} < 1$ in this example isn't an accident. In this situation, even with more than one hole, it is easy to show that $|\textbf{D}|_{p} < 1$.
\bibliographystyle{amsplain}

\begin{thebibliography}{10}
\bibitem{hC2007}
Henri Cohen,
\emph{Number Theory, Vol.~1,2},
Graduate Texts in Mathematics, 239, 240
Springer Verlag, New York, 2007.

\bibitem{jD1977}
Jack Diamond,
\emph{The p-adic Log Gamma Function and p-adic Euler Constants},
Trans. Amer. Math. Soc., \textbf{233}(1977), 321--337

\bibitem{nK1980}
Neal Koblitz,
\emph{p-adic Analysis: a Short Course on Recent Work},
London Mathematical Society Lecture Note Series 46,
Cambridge University Press, Cambridge, 1980.

\bibitem{aR2000}
Alain~M. Robert,
\emph{A Course in p-adic Analysis},
Graduate Texts in Mathematics, 198,
Springer Verlag, New York, 2000.

\bibitem{wS1984}
W.~H. Schikhof,
\emph{Ultrametric calculus, an introductin to p-adic analysis},
Cambridge studies in advanced mathematics,  4,
Cambridge University Press, Cambridge, 1984.

\bibitem{lS1938}
L.~G. Shnirel'man,
\emph{On Functions in Normed, Algebraically Closed Division Rings},
Izvestiya AN SSSR \textbf{2} (1938), 487--498.


\end{thebibliography}

\end{document}